\documentclass[a4paper,10pt]{article}

\usepackage{amsthm}
\usepackage{amsmath}
\usepackage{amsxtra}
\usepackage{amssymb}
\usepackage{turnstile}
\usepackage{enumitem}
\usepackage{undertilde}
\usepackage{xcolor}
\usepackage[retainorgcmds]{IEEEtrantools}
\usepackage{geometry}
\usepackage{thmtools}
\usepackage{hyperref}
\usepackage{mathrsfs}
\usepackage{changepage}

\declaretheorem[name=Definition,style=definition,qed=$\dashv$,
numberwithin=section]{dfn}

\declaretheorem[name=Fact,style=plain,sibling=dfn]{fact}
\declaretheorem[name=Theorem,style=plain,sibling=dfn]{tm}

\declaretheorem[name=Lemma,style=plain,sibling=dfn]{lem}

\declaretheorem[name=Remark,style=definition,sibling=dfn]{rem}
\declaretheorem[name=Claim,style=plain]{clm}
\declaretheorem[name=Claim,style=plain,numbered=no]{clm*}
\declaretheorem[name=Sublaim,style=plain,numbered=no]{sclm*}
\declaretheorem[name=Sublaim,style=plain,sibling=dfn]{sclm}
\declaretheorem[name=Case,style=definition]{case}
\declaretheorem[name=Subcase,style=definition]{scase}
\declaretheorem[name=Stage,style=definition,numbered=no]{stage*}

\newcommand{\Sat}{\mathrm{Sat}}
\newcommand{\cone}{\mathrm{c}}
\newcommand{\rung}{\theta}
\newcommand{\RR}{\mathbb R}
\newcommand{\PP}{\mathbb P}
\newcommand{\BB}{\mathbb B}
\newcommand{\sub}{\subseteq}
\newcommand{\cross}{\times}
\newcommand{\all}{\forall}

\newcommand{\om}{\omega}

\newcommand{\OR}{\mathrm{OR}}

\newcommand{\Hull}{\mathrm{Hull}}
\newcommand{\cut}{\backslash}

\newcommand{\Tt}{\mathcal{T}}
\newcommand{\Ss}{\mathcal{S}}
\newcommand{\Uu}{\mathcal{U}}

\newcommand{\rg}{\mathrm{rg}}

\newcommand{\ins}{\trianglelefteq}

\newcommand{\pins}{\triangleleft}
\newcommand{\npins}{\ntriangleleft}
\newcommand{\crit}{\mathrm{cr}}
\newcommand{\rest}{\!\upharpoonright\!}
\newcommand{\com}{\circ}
\newcommand{\lh}{\mathrm{lh}}

\newcommand{\sats}{\models}

\newcommand{\J}{\SS}
\newcommand{\AD}{\mathrm{AD}}

\newcommand{\DC}{\mathrm{DC}}

\newcommand{\HC}{\mathrm{HC}}
\newcommand{\ZFC}{\mathrm{ZFC}}
\newcommand{\ZF}{\mathrm{ZF}}
\newcommand{\Coll}{\mathrm{Col}}
\newcommand{\es}{\mathbb{E}}
\newcommand{\Pgap}{\Pg}
\newcommand{\ld}{\mathrm{ld}}
\newcommand{\core}{\mathfrak{C}}
\newcommand{\pred}{\mathrm{pred}}
\newcommand{\id}{\mathrm{id}}
\newcommand{\conc}{\ \widehat{\ }\ }
\newcommand{\forces}{\dststile{}{}}

\newcommand{\rSigma}{\mathrm{r}\Sigma}
\newcommand{\mSigma}{\mathrm{m}\Sigma}
\DeclareMathOperator{\Th}{Th}

\DeclareMathOperator{\card}{card}
\DeclareMathOperator{\cof}{cof}
\DeclareMathOperator{\wfp}{wfp}
\DeclareMathOperator{\rank}{rank}
\newcommand{\OD}{\mathrm{OD}}

\newcommand{\bfrSigma}{\utilde{\rSigma}}

\newcommand{\cHull}{\mathrm{cHull}}
\newcommand{\M}{\mathcal{M}}

\newcommand{\tu}{\textup}
\newcommand{\stack}{\mathrm{stack}}

\newcommand{\Lim}{\mathrm{Lim}}
\newcommand{\Lp}{\mathrm{Lp}}
\newcommand{\alphagap}{{\alpha_g}}
\newcommand{\betagap}{{\beta_g}}
\newcommand{\Gammagap}{{\Gamma_g}}
\newcommand{\WO}{\mathrm{WO}}

\newcommand{\dropset}{\mathscr{D}}
\newcommand{\vareps}{\varepsilon}
\newcommand{\Pg}{{P_{\mathrm{g}}}}
 \newcommand{\steelurl}{http://math.berkeley.edu/\textasciitilde steel}
 \newcommand{\Jj}{\mathcal{J}}
\renewcommand{\SS}{\mathbb{S}}
\newcommand{\wh}{\widehat}
\newcommand{\ladder}{\mathrm{ld}}

\newcommand{\wt}{\widetilde}
\begin{document}
\title{Ladder mice}
\author{Farmer Schlutzenberg\footnote{
This research was funded by the Austrian Science Fund (FWF) [10.55776/Y1498].
}\\
farmer.schlutzenberg@tuwien.ac.at\\
TU Wien}

\maketitle

\begin{abstract}
Assume ZF + AD + $V=L(\RR)$. We prove some ``mouse set'' theorems, for definability over $\Jj_\alpha(\RR)$ where $[\alpha,\alpha]$ is a projective-like gap (of $L(\RR)$) and $\alpha$ is either a successor ordinal or has countable cofinality, but  $\alpha\neq\beta+1$ where $\beta$ ends a strong gap.
For such ordinals $\alpha$ and  integers $n\geq 1$,
 we show that there is a mouse $M$ with $\RR\cap M=\OD_{\alpha n}$.

The proof involves an analysis of ladder mice
and their generalizations to $\Jj_\alpha(\RR)$.
This analysis is related to earlier work of Rudominer, Woodin and Steel on ladder mice.
However, it also yields a new proof of the mouse set theorem even at the least point where ladder mice arise -- one which avoids the stationary tower. The analysis also yields a corresponding ``anti-correctness'' result on a cone, generalizing facts  familiar in the projective hierarchy;  for example, that $(\Pi^1_3)^V\rest M_1$ truth is $(\Sigma^1_3)^{M_1}$-definable and $(\Sigma^1_3)^{M_1}$ truth is $(\Pi^1_3)^V\rest M_1$-definable.

We also define and study versions   of ladder mice on a cone at the end of weak gap, and at the successor of the end of a strong gap,
and an anti-correctness result on a cone there.
\end{abstract}

\tableofcontents

\section{Introduction}
Assume ZF + $\AD^{L(\RR)}$, so by \cite{kechrisADDCLR}, DC holds. By a result of Woodin \cite[Theorem 1.5]{twms}, if $\alpha$ is a limit ordinal\footnote{We use Jensen's $\Ss$-hierarchy, not his $\Jj$-hierarchy, throughout; $\Jj_\alpha(\RR)=\Ss_{\alpha\om}(\RR)$ for all ordinals $\alpha$. This diverges from \cite{twms}. An \emph{$\Ss$-gap} is just an interval $[\gamma\om,\delta\om]$ such that $[\gamma,\delta]$ is a gap. In \cite[Theorem 1.3]{twms},
it appears that $\alpha$ is assumed to be a limit ordinal in order to avoid the possibility that $\Jj_\alpha(\RR)$ ends a strong gap.
Note that because of this assumption, even in the case of \cite[Theorem 1.5]{twms} that $\mathfrak{P}_\alpha(\RR)\sub L(\RR)$,
\cite[Theorem 1.5]{twms} is more general than \cite[Theorem 1.3]{twms}. And note that $\mathfrak{P}_\alpha(\RR)$ is closed under complement.} which starts an $\Ss$-gap $[\alpha,\beta]$ and $\alpha$ is not of form $\gamma+\om$ where $\gamma$ ends a strong $\Ss$-gap, then $\OD^{<\alpha}$ is a mouse set. Here $\OD^{<\alpha}=\OD^{<\alpha}(\emptyset)$, where $\OD^{<\alpha}(x)$ is the set of reals $y$ such that for some ordinal $\beta<\alpha$ which is either a limit or $0$, $y$ is definable over $\Ss_\beta(\RR)$
from countable ordinals in the codes;
that is, there is some $\xi<\om_1$ and some formula $\varphi$ such that for all $w\in\WO_\xi$ (that is, $w$ is a wellorder of $\om$ in ordertype $\xi$) and all reals $z$, we have
\begin{equation}\label{eqn:def_y} z=y\iff\Ss_\beta(\RR)\sats\varphi(x,z,w).\end{equation}And a \emph{mouse set}
is the set of reals of some  $(0,\om_1+1)$-iterable premouse $M$.

Rudominer conjectured in \cite[Conjecture 2.24]{rudo_mouse_sets}
that a sharper fact holds, proved certain instances of this conjecture, and posed related questions \cite[p.~18]{rudo_mouse_sets}.
 Steel also refers to such a conjecture in \cite[prior to Theorem 1.3]{twms};
 see also the introduction of \cite{rudo_steel}. Given an ordinal $\beta$ which is either a limit or $0$,  $n\in[1,\om)$
and $x\in\RR$, let $\OD_{\beta n}(x)$ denote the set of reals $y$ such that for some $\Sigma_n$ formula (of the $L(\RR)$ language) and some $\xi<\om_1$,
for all $w\in\WO_\xi$ and
all reals $z$, line (\ref{eqn:def_y}) holds.
Write $\OD_{\beta,n}=\OD_{\beta,n}(\emptyset)$.

Define the pointclasses $\Sigma^\RR_n$ by setting $\Sigma^\RR_1=\Sigma_1$
(in the $L(\RR)$ language),
$\Pi^\RR_n=\neg\Sigma^\RR_n$,
and $\Sigma^\RR_{n+1}=\exists^\RR\Pi^\RR_n$.
Let $\OD^\RR_{\beta n}(x)$ denote the set of reals $y$ such that for some $\Sigma^\RR_n$ formula and some $\xi<\om_1$,
for all $w\in\WO_\xi$ and all reals $z$,
line (\ref{eqn:def_y}) holds.\footnote{Ignoring the difference in indexing, in Rudominer's notation, $A_{(\beta n)}=\OD^\RR_{\beta,n+1}$,
at least assuming $\AD^{L(\RR)}$, as we are.
Note that we have defined $\OD^\RR_{\beta,n+1}$ slightly differently to Rudominer's definition of $A_{(\beta,n)}$, because
we have required directly by definition that $y$ is defined from each $w\in\WO_\xi$, uniformly in $w$, whereas this uniformity is not part of Rudominer's definition; but cf.~\cite[p.~18]{rudo_mouse_sets}, where it is shown that this makes no difference.}
Rudominer conjectured that if $[\beta,\beta]$ is a projective-like $\Ss$-gap which is not the successor of a strong $\Ss$-gap then $\OD^\RR_{\beta,n+1}$
is a mouse set, for each $n\geq 0$.
He also asked  \cite[p.~18]{rudo_mouse_sets}, for $[\beta,\beta]$ a projective-like $\Ss$-gap and $n\geq 1$,
whether $\OD_{\beta,n+1}=\OD^{\RR}_{\beta,n+1}$ (the equality holds when $n=0$ directly by definition).

In this paper we consider these questions in the case that $\beta$ is an ordinal such that $\cof^{L(\RR)}(\beta)=\om$
and $[\beta,\beta]$ is a projective-like $\Ss$-gap which is not the successor of a strong $\Ss$-gap.\footnote{Recall that since $\beta$ is a limit ordinal,
$\Ss_\beta(\RR)=\Jj_{\gamma}(\RR)$
where $\gamma\om=\beta$,
and $\gamma$ could be a successor ordinal here.}
Rudominer established only partial results on this case in \cite{rudo_mouse_sets} (along with further such results in other cases).
Results of Rudominer from \cite{rudo_mouse_sets},
combined with more recent work of Woodin \cite{rudo_ladder},
established that $\OD_{\omega 2}=\OD^\RR_{\omega 2}$ is a mouse set. Adapting the methods of these papers,
combined with some from \cite{twms} and \cite{gaps_as_derived_models}, we will show the following:

\begin{tm}\label{tm:lightface_mouse_set}
Assume $\ZF+\AD^{L(\RR)}$.
Let $[\beta,\beta]$
be an $\Ss$-gap of $L(\RR)$
such that $\cof^{L(\RR)}=\om$
and $\beta$ is not of the form $\gamma+\om$ where $\gamma$ ends a strong $\Ss$-gap. Let $n\in[1,\om)$. Then:
\begin{enumerate}[label=--]
 \item $\OD_{\beta n}=\OD^\RR_{\beta n}$, and
 \item there is an $\om$-small $(0,\om_1+1)$-iterable premouse $M$ with
 $\OD_{\beta n}=\RR\cap M$.
\end{enumerate}
\end{tm}
(This ``mouse set'' result was already known for $n=1$, and in fact then $\OD_{\beta 1}=\OD^\RR_{\beta 1}=\OD_{<\beta}$. In \cite{rudo_ladder}, the result $\beta=\om$ and $n=2$. His proof easily extends to certain other such values of $\beta$ and $n=2$, but doesn't appear to work for all of them.)

Note that the first conclusion stated in the theorem -- that $\OD_{\beta n}=\OD^\RR_{\beta n}$ -- is a purely descriptive fact,
with no mention of mice in the hypothesis or conclusion. But this equality
will be established by finding a mouse $M$ with $\OD_{\beta n}\sub \RR\cap M\sub \OD^\RR_{\beta n}$.

Central to our analysis is a well known  phenomenon, which we call \emph{anti-correctness}.
Recall that given any $\Pi^1_1$ formula $\varphi$, there is a $\Sigma_1$ formula $\psi_{\Sigma_1,\varphi}$
(of the language of set theory)
and also a $\Sigma^1_1$ formula $\psi_{\Sigma^1_1,\varphi}$
such that for all reals $x\in L_{\om_1^{\mathrm{ck}}}$, we have
\[ \varphi(x)\iff L_{\om_1^{\mathrm{ck}}}\sats\psi_{\Sigma_1,\varphi}(x)\iff L_{\om_1^{\mathrm{ck}}}\sats\psi_{\Sigma^1_1,\varphi}(x),\]
and the natural maps $\varphi\mapsto\psi_\varphi$
and $\varphi\mapsto\varrho_\varphi$
are recursive.
(Given $\psi_{\Sigma_1,\varphi}$,
take $\psi_{\Sigma^1_1,\varphi}(x)$
to assert that there is a real which codes a model $M$ for the language of set theory such that $M$ is $\om$-wellfounded and $M\sats$ ``$V=L\wedge \psi_{\Sigma_1,\varphi}(x)$''. This works by Ville's Lemma.)
Secondly, given any $\Pi^1_1$ formula $\varphi$, there is a $\Sigma^1_1$ formula $\varrho_{\Sigma^1_1,\varphi}$
such that for all reals $x\in L_{\om_1^{\mathrm{ck}}}$, we have
\[ L_{\om_1^{\mathrm{ck}}}\sats\varphi(x)\iff \varrho_{\Sigma^1_1,\varphi}(x).\]
(Take $\varrho_{\Sigma^1_1,\varphi}(x)$
to assert that there is an $\om$-wellfounded model $M$ for the language of set theory which models KP + ``$\varphi(x)$ + $V=L$ + there is no ordinal $\alpha$ such that $L_\alpha\sats$ KP''.
Again,  this works by Ville's Lemma.) This is ``anti-correctness'', since not only is $L_{\om_1^{\mathrm{ck}}}$ not $\Sigma^1_1$-correct\footnote{For a pointclass $\Gamma$, a model $M$ is  \emph{$\Gamma$-correct} iff  $M\sats\varphi(x)\iff V\sats\varphi(x)$ for all
formulas $\varphi$ in $\Gamma$  and reals $x\in M$.}, but $(\Pi^1_1)^V\rest (\RR\cap L_{\om_1^{\mathrm{ck}}})$
is $(\Sigma^1_1)^{L_{\om_1^{\mathrm{ck}}}}$,
and $(\Pi^1_1)^{L_{\om_1^{\mathrm{ck}}}}$
is $(\Sigma^1_1)^V$.

The analogue holds two real quantifiers higher. Given any $\Pi^1_3$ formula $\varphi$, there is a $\Sigma_1$ formula $\psi_{\Sigma_1,\varphi}$ in the language of passive premice, and also a $\Sigma^1_3$ formula $\psi_{\Sigma^1_3,\varphi}$,
such that for all reals $x\in M_1$, we have
\[ \varphi(x)\iff M_1\sats\psi_{\Sigma_1,\varphi}(x)\iff M_1\sats\psi_{\Sigma^1_3,\varphi}(x).\]
(The formula $\psi_{\Sigma_1,\varphi}(x)$ asserts that there is some proper segment of $M_1$ which is a $\varphi(x)$-prewitness,
where
a \emph{$\varphi(x)$-prewitness}
is a premouse $N$ with $x\in N$
and such that there is $\delta\in\OR^N$ such that $N\sats\ZF^-$ + ``$\delta$ is Woodin
and the extender algebra $\BB_\delta$
at $\delta$ forces $\varphi(x)$''.
And the formula $\psi_{\Sigma^1_3,\varphi}(x)$ asserts
``there is a countable $1$-small $\Pi^1_2$-iterable $\varphi(x)$-prewitness''. This works by the analogue of Ville's Lemma,
by which if $N$ is $\Pi^1_2$-iterable but $N$ is not iterable, then $M_1|\om_1^{M_1}\ins N$.)
Secondly, given any $\Pi^1_3$ formula $\varphi$,
there is a $\Sigma^1_3$ formula 
$\varrho_{\Sigma^1_3,\varphi}$
such that for all reals $x\in M_1$, we have
\[ M_1\sats\varphi(x)\iff \varrho_{\Sigma^1_3,\varphi}(x).\]
(Take $\varrho_{\Sigma^1_3,\varphi}(x)$
to assert that there is a countable, $\Pi^1_2$-iterable premouse $N$ which is a putative $M_1^\#$ and $N\sats\varphi(x)$. Again, this works by the analogue of Ville's Lemma.)

This generalizes in a direct manner to  $\Pi^1_{2n+1},\Sigma^1_{2n+1}$ and the mouse $M_{2n-1}$, for all integers $n>1$.

Now the first pointclass beyond projective which is analogous
to $\Pi^1_{2n+1}$, is $\Pi_2^{\Ss_{\om}(\RR)}$
(here $\Ss_{\om}(\RR)=\Jj(\RR)$ is the rud-closure of $\RR\cup\{\RR\}$). Our understanding of this pointclass (and its dual, $\Sigma_2^{\Ss_\om(\RR)}$)
was advanced significantly
with Woodin's proof of the mouse set theorem at this level, via his analysis of ladder mice in
 \cite[\S\S1--5]{rudo_ladder}.
Woodin's analysis was extended by Steel in his unpublished notes \cite{steel_ladder} (part of which forms  \cite[\S7]{rudo_ladder}, but involves also more not included in \cite{rudo_ladder}). It seems to the author that anti-correctness at this level was intended to be conjectured in \cite[\S7]{rudo_ladder}, but what is written there is actually something else.\footnote{The version for $\Sigma^1_3,\Pi^1_3$ stated in \cite{rudo_ladder} is also different to anti-correctness. This all refers to the arXiv v2 preprint version of the paper.}
We establish anti-correctness at this level of $L(\RR)$,
and also other levels. Recall that $M_{\ladder}$
is the minimal \emph{ladder mouse}
(see \cite{rudo_ladder}),
and $M_{\ladder}$ is to $\Pi_2^{\Ss_\om(\RR)}$ as $L_{\om_1^{\mathrm{ck}}}$ is to $\Pi^1_1$ and $M_1$ is to $\Pi^1_3$:

\begin{tm}
There are recursive maps $\varphi\mapsto\psi_{\Sigma_2,\varphi}$ and $\varphi\mapsto\varrho_{\Sigma_2,\varphi}$
 such that for all $\Pi_2$ formulas $\varphi$,
 $\psi_{\Sigma_2,\varphi}$ and $\varrho_{\Sigma_2,\varphi}$ are $\Sigma_2$ formulas
 \tu{(}where the formulas are in the  $L(\RR)$ language\footnote{Recall that this is the language of set theory augmented with a constant symbol $\dot{\RR}$, which is interpreted as the set of reals of the structure.} and the complexity is in this sense\tu{)}
 and for all reals $x\in M_{\ladder}$, we have:
 \begin{enumerate}
 \item $\Ss_\om(\RR)\sats\varphi(x)\iff (\Ss_\om(\RR))^{M_{\ladder}}\sats\psi_{\Sigma_2,\varphi}(x)$, and

 \item $(\Ss_\om(\RR))^{M_{\ladder}}\sats\varphi(x)\iff \Ss_\om(\RR)\sats\varrho_{\Sigma_2,\varphi}(x)$.
 \end{enumerate}
\end{tm}

The formula $\psi_{\Sigma_2,\varphi}(x)$ will assert ``there is a $\Pi_1$-iterable $\varphi(x)$-prewitness'', where we will formulate the appropriate notion of \emph{$\varphi(x)$-prewitness}
in Definition \ref{dfn:varphi(x)-prewitness_tree_version} below, and ``$\Pi_1$-iterability'' is recalled in Definition \ref{dfn:Pi_1-iterable}. The definition of $\Pi_1$-iterability is standard,
but the notion of \emph{$\varphi(x)$-prewitness} seems to be new.

The instance of the mouse set theorem for $\OD_{\omega 2}$,
which was established by Woodin's result, is also a corollary of the theorem above. The proof does not depend on Woodin's result,
and our proof -- although related to Woodin's -- does not make any mention of stationary tower forcing,
whereas this seemed to be an important component of Woodin's proof.

The proof actually goes through more generally, and is hardly different in the more general context, so we give it there.

\begin{dfn}\label{dfn:beta-ascending}
 We say that $\beta\in\OR$ is  \emph{$\om$-standard}
 iff $[\beta,\beta]$
 is a projective-like $\Ss$-gap
 of $L(\RR)$ (so $\beta$ is a limit ordinal),
  $\cof^{L(\RR)}(\beta)=\om$,
 and $\beta$ is not of form $\gamma+\om$ where $\gamma$ ends a strong $\Ss$-gap of $L(\RR)$.
 
 Let $\beta$ be $\om$-standard.
 A pair $(x_{\cof},\varphi_{\cof})$
 is called \emph{$\beta$-ascending}
 iff $x_{\cof}\in\RR$, $\varphi_{\cof}$ is a $\Sigma_1$ formula of the $L(\RR)$  language,
 \[ \Ss_{\beta}(\RR)\sats\all^\om n\ \varphi_{\cof}(n,x_{\cof}),\]
 and letting $\alpha_n$ be the least $\alpha$ such that
 \[ \Ss_{\alpha}(\RR)\sats\varphi_{\cof}(n,x_{\cof}),\]
 then $\alpha_n<\alpha_{n+1}<\beta$ for all $n<\om$, and $\sup_{n<\om}\alpha_n=\beta$, and finally,
 if $\beta=\gamma+\om$
 where $\gamma$ is a limit or $\gamma=0$, then $\gamma<\alpha_0$.
\end{dfn}

Note that the ordinals $\alpha_n$
are allowed to be successors, and $\beta$ might be of form $\beta=\gamma+\om$ for some $\gamma$. We have:
\begin{fact}
 Assume $\DC_\RR$
 and let $[\beta,\beta]$ be a projective-like $\Ss$-gap with $\cof^{L(\RR)}(\beta)=\om$.
 Then there is a $\beta$-ascending pair $(x_{\cof},\varphi_{\cof})$.
\end{fact}

Since $\AD^{L(\RR)}$ implies DC, we are free to apply this fact.
The following result describes the direct analogue of ladder mice just beyond $[\beta,\beta]$:

\begin{tm}
Assume $\ZF+\AD^{L(\RR)}$. Let $\beta$ be $\om$-standard.
Let $(x_{\cof},\varphi_{\cof})$ be $\beta$-ascending, and let $\left<\alpha_n\right>_{n<\om}$
be as in Definition \ref{dfn:beta-ascending}.

Write $\Gamma=\Sigma_1^{\Ss_\beta(\RR)}$.
Let $y\in\RR$. Then there is a minimal sound $(y,x_{\cof})$-mouse
$M=M^\Gamma_{\ladder}(y,x_{\cof})$
such that:
\begin{enumerate}
 \item $M$ has no largest cardinal
 (so $\rho_1^M=\OR^M$),
 \item $\rho_2^M=\om$ and $p_2^M=\emptyset$,
 \item For each $n<\om$
 there is an $M$-cardinal $\rung_n^M<\OR^M$ such that $M\sats$``$\rung_n^M$ is not Woodin'', and letting
 $Q_n\pins M$
 be the Q-structure for $M|\rung_n^M$, there is no above-$\rung_n^M$ iteration strategy $\Sigma$ for $Q_n$ with $\Sigma\in\Ss_{\alpha_n}(\RR)$.\footnote{But $Q_n$ will be above-$\rung_n^M$-iterable via a strategy in $\Ss_\beta(\RR)$.}
\end{enumerate}
\end{tm}

By \cite[***Lemma 3.1]{gaps_as_derived_models}, we have:
\begin{fact}\label{fact:Sigma_1-elem_pi}
 Assume $\ZF+\AD^{L(\RR)}$.
 Let $\beta$ be $\om$-standard and
 $(x_{\cof},\varphi_{\cof})$ be 
 $\beta$-ascending. Let $\Gamma=\Sigma_1^{\Ss_\beta(\RR)}$, $y\in\RR$ and $M=M^{\Gamma}_{\ladder}(y,x_{\cof})$. Then there is a unique pair $(\bar{\beta},\pi)$ such that  $\bar{\beta}\in\OR$
 and
 \[ \pi:\Ss_{\bar{\beta}}(\RR^M)\to\Ss_\beta(\RR) \]
 is $\Sigma_1$-elementary.
 Moreover,
 $\RR^M=\RR\cap\Ss_{\bar{\beta}}(\RR^M)$
 and $\pi\rest\RR^M=\id$.
\end{fact}

Write $\bar{\beta}_{y,x_{\cof}}=\bar{\beta}$, where $\bar{\beta}$
is as above.

\begin{tm}\label{tm:anti-correctness_om-standard}
 Assume $\ZF+\AD^{L(\RR)}$.
 Let $\beta$ be $\om$-standard and
 $(x_{\cof},\varphi_{\cof})$ be 
 $\beta$-ascending. Write $\Gamma=\Sigma_1^{\Ss_\beta(\RR)}$.
 Then there are recursive maps $\varphi\mapsto\psi_{\Sigma_2,\varphi}$ and $\varphi\mapsto\varrho_{\Sigma_2,\varphi}$ such that for all $\Pi_2$ formulas $\varphi$, $\psi_{\Sigma_2,\varphi}$ and $\varrho_{\Sigma_2,\varphi}$ are $\Sigma_2$ formulas \tu{(}all with respect to the $L(\RR)$ language\tu{)} and for all $y\in\RR$, letting $\bar{\beta}=\bar{\beta}_{y,x_{\cof}}$, for  all $x\in \RR\cap M_{\ld}^\Gamma(x_{\cof},y)$, we have:
 \begin{enumerate}[label=\tu{(}\alph*\tu{)}]
  \item $\Ss_{\beta}(\RR)\sats\varphi(x)\iff\Ss_{\bar{\beta}}(\RR^M)\sats\psi_{\Sigma_2,\varphi}(x,x_{\cof})$, and
  \item $\Ss_{\bar{\beta}}(\RR^M)\sats\varphi(x)\iff\Ss_\beta(\RR)\sats\varrho_{\Sigma_2,\varphi}(x,x_{\cof})$.
 \end{enumerate}
\end{tm}

Finally, in \S\ref{sec:admissible_gaps},
we will identify, on a cone, ladder mice at the ends of weak gaps and at successors of the ends of strong gaps, prove an associated anti-correctness result.\footnote{\S\ref{sec:admissible_gaps}
in particular needs some more explanation added.}

\section{Projective-like gaps of countable cofinality}

\subsection{The mouse set theorem on an explicit cone}\label{sec:coarse_countable_cof}

Fix an $\om$-standard
ordinal $\alphagap$ (see Definition \ref{dfn:beta-ascending}). In particular, $\Ss_\alphagap(\RR)=\Hull_1^{\Ss_\alphagap(\RR)}(\RR)$ and  for each $\Sigma_1$ formula $\psi$ and each real $x$, 
$\Ss_\alphagap\sats\psi(x)$ iff there is a $\psi(x)$-witness which is iterable via a strategy in $\Ss_\alphagap(\RR)$ (cf.~\cite{gaps_as_derived_models}).
Let $(x_{\cof},\varphi_{\cof})$
be $\alphagap$-ascending.

\begin{lem}\label{lem:scales_after_weak_gap}
Suppose that $\alphagap=\beta+\om$
where $\beta$ ends a weak gap $[\alpha,\beta]$. Then there is a real $x_{\mathrm{sc}}$
from which the standard scale $\vec{\leq}$ at the end of the weak gap $[\alpha,\beta]$ \tu{(}cf.~\cite{scales_in_LR}\tu{)}
is definable over $\Ss_\beta(\RR)$ from $x_{\mathrm{sc}}$,
and $x_{\mathrm{sc}}$ is $\Delta_1^{\Ss_\alphagap(\RR)}(\{x_{\cof}\})$.
\end{lem}
\begin{proof}
Fix $n<\om$. Then there is an $x_{\cof}$-premouse $N$ and $\delta<\OR^N$ such that $N$ is pointwise definable, $N\sats\ZF^-$ + ``$\delta$ is Woodin'',
$N$ is an above-$\delta$, $\varphi_{\cof}(n,x_{\cof})$-prewitness, and $N$ is $(\om,\om_1)$-iterable via a strategy $\Sigma\in \Ss_{\alphagap}(\RR)$. The least such $\{N\}$ is $\Delta_1^{\Ss_{\alphagap}(\RR)}(\{x_{\cof}\})$. By taking $n$ large enough, note that
 $\Coll(\om,\delta)$ forces over $N$ that the standard scale $\vec{\leq}$ at the end of the weak gap $[\alpha,\beta]$ is definable over $\Ss_\beta(\RR)$ from the generic real $\dot{x}$ as parameter, meaning more precisely that this fact is expressed by a statement in the (partial) theory of a level of $L(\RR)$ for which $N^{\Coll(\om,\delta')}$ has a universally Baire representation, where $\delta'$ is the least Woodin of $N$ such that $\delta'>\delta$.
But there is also an $(N,\BB_\delta^N)$-generic filter $g$
such that
$p\in g$ and $\{g\}$ is $\Delta_1^{\Ss_{\alphagap}(\RR)}(\{N\})$, and hence $\{g\}$ is $\Delta_1^{\Ss_{\alphagap}(\RR)}(\{x_{\cof}\})$. So we get
a real $x_{\mathrm{sc}}\in N[g]$
as desired.
\end{proof}

\begin{dfn}\label{dfn:T_n}
As in Definition \ref{dfn:beta-ascending}, let $\alpha_n$ be the least $\alpha$ such that $\Ss_{\alpha}(\RR)\sats\varphi_{\cof}(n,x_{\cof})$.
Let
\[ A_n=\big\{(\varphi,x,y)\bigm|\varphi\text{ is  a }\Pi_1\text{ formula }\wedge x,y\in\RR\wedge\Ss_{\alpha_n}(\RR)\sats\varphi(x,y,z)\big\}.\]
Let $\vec{\Phi}_n$ be the  natural very good scale on $A_n$.\footnote{See \cite{scales_in_LR}
and Lemma \ref{lem:scales_after_weak_gap}.}
In particular, $\vec{\Phi}_n\in\Ss_{\alphagap}(\RR)$
and $\left<\vec{\Phi}_n\right>_{n<\om}$ is $\Sigma_1^{\Ss_{\alphagap}(\RR)}(\{x_{\cof}\})$. Let $\vec{\leq}_n$ be the sequence of prewellorders associated to $\vec{\Phi}_n$. Write $\vec{\Phi}_n=\left<\Phi_{ni}\right>_{i<\om}$ and $\vec{\leq}_n=\left<{\leq}_{ni}\right>_{i<\om}$.
Take $\Phi_{ni}$ regular
(its range is an ordinal),
let $\kappa_{ni}=\rg(\Phi_{ni})$,
and $\kappa_n=\sup_{i<\om}\kappa_{ni}$. Let $T_n$ be the tree of $\vec{\Phi}_n$.
\end{dfn}

\begin{dfn}
Given an ordinal $\alpha\leq\alphagap$, we say that a putative iteration tree $\Tt$ on a countable tame premouse $N$ is \emph{$\Ss_{\alpha}(\RR)$-guided}
iff for every limit ordinal $\lambda<\lh(\Tt)$,
there is $Q\ins M^\Tt_\lambda$
such that $Q$ is a $\delta(\Tt\rest\lambda)$-sound Q-structure for $M(\Tt\rest\lambda)$ and $Q$ is above-$\delta(\Tt)$,
$(k,\om_1)$-iterable via a strategy $\Sigma\in\Ss_{\alpha}(\RR)$
(in the codes, if $\alpha<\om_1$),
where $k$ is least such that $\rho_{k+1}^Q\leq\delta(\Tt)$.
\end{dfn}

\begin{dfn}\label{dfn:Pi_1-iterable}
Let $N$ be a tame $n$-sound premouse, where $n<\om$. We say that $N$ is \emph{$\Pi_1^{\alphagap}$-$(n,\om_1)$-iterable} iff for every putative $n$-maximal tree $\Tt$ on $N$ with $\Tt\in\HC$,
and every ordinal $\alpha<\alphagap$, if $\Tt$ is $\Ss_\alpha(\RR)$-guided then:
\begin{enumerate}[label=--]\item  $\Tt$ has wellfounded models, and
 \item if $\Tt$ has limit length and there is $Q\in\HC$ such that $Q$ is a  $\delta(\Tt)$-sound Q-structure for $M(\Tt)$ and $Q$ is $(k,\om_1)$-iterable via a strategy $\Sigma\in\Ss_\alpha(\RR)$,
 where $k$ is least such that $\rho_{k+1}^Q\leq\delta(\Tt)$,
 then there is a $\Tt$-cofinal branch $b$ such that $Q\ins M^\Tt_b$.\footnote{We will only apply this in case $N$ is a projecting sound premouse, in which case it is reasonable to hope for such a branch $b$.}
\end{enumerate}

If $N$ is tame and $\om$-sound,
we say that $N$ is \emph{$\Pi_1^{\alphagap}$-iterable}
iff $N$ is $\Pi_1^{\alphagap}$-$(\om,\om_1)$-iterable.
\end{dfn}

\begin{dfn}
Let $\Gamma=\Sigma_1^{\Ss_\alphagap(\RR)}$. For $\alpha\leq\alphagap$, let $\Gamma_\alpha=\Sigma_1^{\Ss_\alpha(\RR)}$. Given a transitive swo'd\footnote{Recall that \emph{swo'd} abbreviates \emph{self-wellordered}, which says that $X$ has form $(X',W)$ where $X'$ is transitive and $W$ is a wellorder of $X'$. This induces a canonical wellorder on any $X$-premouse, which extends $W$. We could probably make do with there simply being a wellorder of $X$ in $\Jj(X)$, in which case there might not be a \emph{canonical} wellorder of $X$-premice, but there would be some $\vec{x}\in X^{<\om}$ and some wellorder determined by $\vec{x}$. Actually we could just work with arbitrary transitive $X$; we have only avoided this to slightly simplify some  statements to do with cardinalities in $X$-premice, and since it suffices for our purposes.} $X\in\HC$,  $\Lp_{\Gamma_\alpha}(X)$ denotes the stack of all $X$-premice $N$
such that $N$ is sound, projects to $X$, and there is an $(\om,\om_1)$-iteration strategy $\Sigma$ for $N$ with $\Sigma\in\Ss_\alpha(\RR)$.
We write $\Lp_\Gamma=\Lp_{\Gamma_{\alphagap}}$.

Given a transitive  swo'd $X\in\HC$ with $x_{\cof}\in X$, and given $n<\om$, $P_n(X)$
denotes the least $P\pins\Lp(X)$ which is a $\varphi_{\cof}(n,x_{\cof})$-prewitness (as an $X$-premouse, so the relevant Woodin cardinals $\delta$ have $\delta>\rank(X)$, etc).
\end{dfn}

 As usual, $\Lp_\Gamma(X)$
is a sound, passive premouse with no largest proper segment. But because $\alphagap$ has countable cofinality, it can be that $\Lp_\Gamma(X)\not\sats\ZF^-$
(so $\Lp_\Gamma(X)$ could project to $X$),
since the strategies $\Sigma$ witnessing that  the projecting $N\pins\Lp_{\Gamma}(X)$
are in fact valid, could appear cofinally in $\Ss_{\alphagap}(\RR)$,
and so the join of these strategies (which is essentially a strategy for $\Lp_\Gamma(X)$) would then not be in $\Ss_{\alphagap}(\RR)$. The simplest  example is of course $\stack_{n<\om}M_n^\#$.
This also occurs given $x_{\cof}\in X$:

\begin{lem}
 Let $X\in\HC$ be transitive swo'd with $x_{\cof}\in X$.
 Then $\Lp_\Gamma(X)=\stack_{n<\om}P_n(X)$.
\end{lem}
\begin{proof}
 Certainly $P=\stack_{n<\om}P_n(X)\ins\Lp_\Gamma(X)$,
 so suppose $P\pins\Lp_\Gamma(X)$.
 Then there is an iteration strategy $\Sigma\in\Ss_{\alphagap}(\RR)$ for $P$.
 But $t=\Th_{\rSigma_1}^{\Ss_{\alphagap}(\RR)}(\RR)$
 can be computed from $\Sigma$,  so $t\in\Ss_{\alphagap}(\RR)$, which is impossible.
\end{proof}

\begin{rem}\label{rem:pres_P_n(X)_under_hull,it_map}
 The premice of form $P_n(X)$ are preserved under the relevant hulls and iteration maps, since being a $\varphi_{\cof}(n,x_{\cof})$-prewitness is suitably definable. We will use this fact wherever needed, without explicit mention.
\end{rem}

\begin{dfn}
Let $X\in\HC$ be transitive swo'd. A \emph{$\Gamma$-ladder over $X$}
is an $X$-premouse $N$ such that
for each $\xi\in[\rank(X),\OR^N)$, we have $\Lp_{\Gamma}(N|\xi)\pins N$, and
for each $\alpha<\alphagap$
there is $\rung<\OR^N$  such that $\rung$  is an $N$-cardinal and $\rung$ is Woodin in $\Lp_{\Gamma_\alpha}(N|\theta)$.

Suppose also that
 $x_{\cof}\in X$. Note
that an $X$-premouse $N$ is a $\Gamma$-ladder over $X$ iff
for each $\xi<\OR^N$, we have $\Lp_{\Gamma}(N|\xi)\pins N$,
and for each $n<\om$,
there is  $\theta<\OR^N$ such that $\theta$ is an $N$-cardinal, 
$P_n(N|\theta)\sats$``$\theta$ is Woodin''. Suppose $N$ is such. Let $\theta_n^N$ denote the least such $\theta$ (with respect to $n<\om$),
and $P^N_n=P_n(N|\theta)$.
Then the $\theta_n^N$ are called the \emph{rungs} of the ladder.
A $\Gamma$-ladder $N$ over $X$is called \emph{minimal} if no $N'\pins N$ is a $\Gamma$-ladder (in particular, $\sup_{n<\om}\theta_n^N=\OR^N$). We write $M_{\ladder}^\Gamma(X)$
for the unique iterable sound minimal $\Gamma$-ladder over $X$.
 We say that a premouse $N$ is \emph{$\Gamma$-ladder-small} if no $N'\ins N$ is a $\Gamma$-ladder. 
\end{dfn}

\begin{dfn}\label{dfn:Sigma_1^alphagap-correct}
 Let $A\sub\RR$. We say that $A$ is \emph{$\Sigma_1^{\alphagap}$-correct} iff there is an ordinal $\bar{\alpha}$
 and $\pi$ such that $A=\RR\cap\Ss_{\bar{\alpha}}(A)$ and
$\pi:\Ss_{\bar{\alpha}}(A)\to\Ss_{\alphagap}(\RR)$ is $\Sigma_1$-elementary.
\end{dfn}

\begin{rem}\label{rem:Sigma_1^alphagap-correct}
Suppose $A$ is $\Sigma_1^{\alphagap}$-correct.
Note that since $\Ss_{\alphagap}(\RR)=\Hull_1^{\Ss_{\alphagap}}(\RR)$, the witnessing pair $(\bar{\alpha},\pi)$ is uniquely determined, and moreover,
$\Ss_{\bar{\alpha}}(A)=\Hull_1^{\Ss_{\bar{\alpha}}(A)}(A)$.\end{rem}

By
\cite[***Lemma 3.1]{gaps_as_derived_models}, we have:

\begin{lem}\label{lem:Sigma_1^alphagap-correct}
 Let $X\in\HC$ be transitive swo'd with $x_{\cof}\in X$.
 Let $N$ be an $X$-premouse
 such that for all $\xi<\OR^N$,
 we have $\Lp_{\Gamma}(N|\xi)\ins N$. Then:
 \begin{enumerate}
  \item $\RR\cap N$
 is $\Sigma_1^{\alphagap}$-correct.
 \item $\RR\cap N[g]$ is $\Sigma_1^{\alphagap}$-correct
 whenever $g$ is $(N,\PP)$-generic
 for a forcing $\PP\in N$.
 \end{enumerate}
\end{lem}

  The following lemma is established by the obvious adaptation of the proof of the analogous fact for the original ladder mice, as in \cite{rudo_mouse_sets} or \cite{rudo_ladder}.

 \begin{lem}
  Let $X\in\HC$ be transitive swo'd with $x_{\cof}\in X$. Let $M=M_{\ladder}^\Gamma(X)$. Let $N$ be a $\Pi_1^{\alphagap}$-iterable
  $X$-premouse which is sound, with $\rho_\om^N=X$.
  Then either
  $N\pins M|\om_1^M$ or $M|\om_1^{M}\ins N$.
 \end{lem}

We now want to move toward defining \emph{$\varphi(y)$-prewitness}
for a $\Pi_2$ formula $\varphi$ (of the $L(\RR)$ language),
and show that for $x\geq_T x_{\cof}$, and $y\in\RR^{M_{\ladder}^\Gamma(x)}$, 
we have $\Ss_\alpha(\RR)\sats\varphi(y)$
iff there is a $\varphi(y)$-prewitness $N$ such that $N\pins M_{\ladder}^\Gamma(x)$.
The intention is that if $N$ is an iterable $\varphi(y)$-prewitness over $x$
which is $\Gamma$-ladder-small,
but $\neg\varphi(y)$ holds, as witnessed by $w\in\RR$,\footnote{Note that $\neg\varphi$ is $\Pi_2$, as opposed to $\all^\RR\Sigma_1$. But because we can refer to $x$, and hence to $x_{\cof}$, it suffices to deal with $\all^\RR\Sigma_1$ formulas.} then we can run a version of the stationary tower argument in \cite{rudo_ladder}, iterating out $N$ to make $w$ generic over various segments of the iterate,
using this to produce nodes in trees $T_x^{N'[g]}$ for various collapse generics $g$, where $T_x$ is the tree for producing a witness to $\neg\varphi(x)$. However, we will also arrange that the full iteration in fact drops infinitely often along its unique branch, and $N'$ is the common part model, and is a ladder mouse;
so since $N$ is iterable and is a proper segment of $M|\om_1^M$,
we will have reached two contradictions.
At each step of the process of iterating to make $w$ generic at the next ``rung'' of the ladder $N'$ we are producing,
there will be a drop in model,
at a cutpoint of the iteration.

\begin{dfn}
Given $X\in\HC$ with $x_{\cof}\in X$, an $X$-premouse $N$ 
and $n<\om$, we say that $N$ is an \emph{$n$-partial-$\Gamma$-ladder}
iff:
\begin{enumerate}[label=--]\item $\Lp_{\Gamma}(N|\xi)\ins N$ for each $\xi<\OR^N$,
\item there are $\rung_0<\ldots<\rung_n<\OR^N$
such that for each $i\leq n$,  $\rung_i$ is the least ordinal $\rung$ such that $\rung$ is an $N$-cardinal and 
 $P_n(N|\rung)\sats$``$\rung$ is Woodin'',
\item $\rung_n^{+N}<\OR^N$
and $\theta_n^{+N}$ is the largest cardinal of $N$.
\end{enumerate}

Let $N$ be an $n$-partial-$\Gamma$-ladder.
For $i\leq n$, we write $\rung_i^N$ for the least $N$-cardinal $\rung$ such that $\rung$ is Woodin in $P_i(N|\rung)$.
\end{dfn}

\begin{dfn}
Let $\varphi(u,v)$ be a $\Pi_1$ formula of the $L(\RR)$ language, of two free variables $u,v$,
and $\psi(u)$ be $\exists^\RR v\ \varphi(u,v)$.

Then $T_{\psi}$ denotes the tree projecting to \[A_\psi=\big\{x\in\RR\bigm|\Ss_{\alphagap}(\RR)\sats\psi(x)\big\},\]
derived in the standard manner from the sequence $\left<T_n\right>_{n<\om}$
(see Definition \ref{dfn:T_n}).
The details are like for the analogous tree introduced in \cite{rudo_ladder}. That is,
fix a recursive bijection  $f:\om^2\to\om$ such that  $f^{-1}(n)\sub n\times n$ for all $n<\om$,
and write $\ulcorner i,j\urcorner=f(i,j)$.
Also write $\ulcorner i_m,j_m\urcorner = m$.
Now let 
$(s,t,u)\in T_\psi$
iff for some $n<\om$,
we have $s,t:n\to\om$, $u:n\to\OR$,
and for each $m<n$,
letting $(i,j)=(i_m,j_m)$,
 \[a_m=\{\ell\leq m\bigm|\ulcorner i,k\urcorner=\ell\text{ for some }k\},\]
$c_m=\card(a_m)$ and $\tau_m:c_m\to a_m$ be the order-preserving bijection,
we have
\[ (\varphi,s\rest c_m,t\rest c_m,u'_m)\in T_i\]
where $u'_m:c_m\to\OR$ and $u'_m(b)=u(\tau_m(b))$.

Let $T_{\psi,x}$ for the section of $T_\psi$ at $x\in\RR$;
so $x\in A_\psi$ iff $T_{\psi,x}$ is illfounded.\end{dfn}

\begin{dfn}\label{dfn:T-tilde}Let $\psi$
be an $\exists^\RR\Pi_1$ formula of the $L(\RR)$ language.

Let $N$ be an $n$-partial-$\Gamma$-ladder. Let $g$ be $\Coll(\om,\theta_n^N)$-generic. Then $T_\psi^{N[g]}$ denotes the version of $T_\psi$ computed in $N[g]$.
That is, we use
the $\pi^{-1}({\vec{\leq}_n})$
in place of $\vec{\leq}_n$ (for each $n<\om$), where
\[ \pi:\Ss_{\bar{\alpha}}(\RR\cap N[g])\to \Ss_\alphagap(\RR) \]
is 
as in Definition \ref{dfn:Sigma_1^alphagap-correct}. By Lemma \ref{lem:T_indep_of_g} below,
$T_\psi^{N[g]}$ is independent of the choice of $g$,  and hence in $N$. We write $\widetilde{T}^N_\psi=T^{N[g]}_\psi$.
We write $\widetilde{T}_{\psi,x}^{N}$ for the section of $\widetilde{T}_\psi^N$ at $x\in\RR\cap N$.
If $\psi$ is clear from context we may just write $\widetilde{T}^N$ and $\widetilde{T}^N_x$ instead.

Let $N'$ be an $(n+1)$-partial-$\Gamma$-ladder with $N|\theta_n^{+N}=N'|\theta_n^{+N'}$;  note that $\theta_i^N=\theta_i^{N'}$ for all $i\leq n$. Then $\widetilde{\pi}^{NN'}_{\psi}:\widetilde{T}^N_\psi\to\widetilde{T}^{N'}_\psi$ denotes the map $\widetilde{\pi}$ given by setting \[\widetilde{\pi}((s,t,(\alpha_0,\ldots,\alpha_{k-1})))=(s,t,(\beta_0,\ldots,\beta_{k-1})) \]
where if $g$ is $(N,\Coll(\om,\theta_n^N))$-generic
and $g'$ is $(N',\Coll(\om,\theta_{n+1}^{N'}))$-generic
with $g\in N'[g']$, and $m<k$ and $x\in A_{i_m}^{N[g]}=A_{i_m}\cap N[g]$ and $\alpha_m=\Phi_{i_m,j_m}^{N[g]}(x)$, then $
\beta_m=\Phi_{i_m,j_m}^{N'[g']}(x)$. Also by Lemma \ref{lem:T_indep_of_g}, $\widetilde{\pi}$ is independent of $g,g'$,
and hence $\pi\in N'$.
We write $\widetilde{\pi}^{NN'}_{\psi,x}:T_{\psi,x}^{N}\to T_{\psi,x'}^{N'}$ for the corresponding section of $\widetilde{\pi}^{NN'}_\psi$ (where $s=x\rest\lh(s)$). In practice, $\psi$ will be clear from context, and we will just write $\widetilde{\pi}^{NN'}_x$.
\end{dfn}
\begin{lem}\label{lem:T_indep_of_g}
 The tree $T^{N[g]}_{\psi}$
 introduced in Definition \ref{dfn:T-tilde} is independent of the $N$-generic $g\sub\Coll(\om,\theta_n^N)$, and hence $\widetilde{T}^N_\psi=T^{N[g]}_\psi\in N$.
 
 The embedding $\widetilde{\pi}^{NN'}_{\psi}$ is independent of $g,g'$, and hence $\widetilde{\pi}^{NN'}_{\psi}\in N'$.
\end{lem}
\begin{proof}
 This is the fact, due to Hjorth,
 and explained in \cite{rudo_ladder}, about collapse generic extensions hitting the same equivalence classes of thin equivalence relations. 
\end{proof}

\begin{dfn}\label{dfn:varphi(x)-prewitness_tree_version}Let $X\in\HC$ be transitive swo'd with $x_{\cof}\in X$.
Let $N$ be an $X$-premouse, $x\in\RR^N$, and $\Delta,t,u\in N$ with
$(t,u)\in(\om\cross\OR)^{<\om}$.  
Let $\varphi$ be a $\all^\RR\Sigma_1$ formula of the $L(\RR)$ language. We say that $(N,\Delta)$ is a \emph{$(\varphi(x),t,u)$-prewitness}
iff, letting $n=\lh(t,u)$
and $(N_n,t_n,u_n,\Delta_n)=(N,t,u,\Delta)$, then
$N$ is an $n$-partial-$\Gamma$-ladder,
and writing $\widetilde{T}_x=\widetilde{T}_{\neg\varphi,x}$, then
 $(t,u)\in\widetilde{T}_x^{N}$, 
$\Delta\in N$
 is a non-empty tree (set of finite sequences  closed under initial segment) whose elements $\sigma$ have form
\[ \sigma=(\sigma_{n+1},\ldots,\sigma_{n+k}) \]
where $k<\om$ and for $0\leq i\leq k$, $\sigma_{n+i}$ has form
\[ \sigma_{n+i}=(N_{n+i},t_{n+i},u_{n+i},\Delta_{n+i}) \]
(so $\sigma_n=(N,t,u,\Delta)$, but $\sigma_n$ is not actually an element of $\sigma$), and moreover, for each $\sigma\in \Delta$, with $\sigma,k,\sigma_{n+i}$ as above, the following conditions hold:
\begin{enumerate}
\item If $\sigma\neq\emptyset$ then for every $i<k$, we have the following:
\begin{enumerate}\item
$N_{n+i+1}\pins N_{n+i}$,
\item $N_{n+i+1}$ is an $(n+i+1)$-partial ladder,
\item $\rho_1^{N_{n+i+1}}=\rho_\om^{N_{n+i+1}}=\theta_{n+i}^{+N_{n+i}}$
(so note $\theta_j^{N_{n+i+1}}=\theta_j^{N_{n+i}}$ for $j\leq n+i$),
\item $(t_{n+i+1},u_{n+i+1})\in \wt{T}_{x}^{N_{n+i+1}}$,
\item $\lh((t_{n+i+1},u_{n+i+1}))=n+i+1$,
\item $\pi_x^{N_{n+i},N_{n+i+1}}((t_{n+i},u_{n+i}))\ins(t_{n+i+1},u_{n+i+1})$,
\item $\Delta_{n+i+1}\in N_{n+i+1}$
is a  tree (set of finite sequences closed under initial segment),
\item 
$ (\sigma_{n+i+2},\ldots,\sigma_{n+k})\in \Delta_{n+i+1}$.
\end{enumerate}
\item
$\Delta_{n+k}=\{\tau\bigm|\sigma\conc\tau\in \Delta\}$.
\item 
For every
$(t',u')\in \wt{T}_x^{N_{n+k}}$
with $(t_{n+k},u_{n+k})\pins(t',u')$
and $\lh(t',u')=n+k+1$,
there is $\sigma'\in\Delta$
with $\sigma'=\sigma\conc((N',\wt{\pi}_x^{N_{n+k}N'}(t',u'),\Delta'))$
for some $N',\Delta'$.
\end{enumerate}

A \emph{$\varphi(x)$-prewitness}
is a $(\varphi(x),\emptyset,\emptyset)$-prewitness.
\end{dfn}

\begin{dfn}
Let $X\in\HC$ be transitive swo'd with $x_{\cof}\in X$.
Let $M=M^{\Gamma}_{\ladder}(X)$.
 Let $\psi$ be an $\exists^\RR\Pi_1$ formula. Then $\wt{T}_\psi^{M}$ denotes
 the direct limit of the trees $\wt{T}^{N_n}_\psi$,
 under the maps $\wt{\pi}^{N_nN_{n+1}}_\psi$,
 where $N_n=M|\theta_n^{++M}$.  And $\wt{\pi}^{N_nM}:\wt{T}^{N_n}_\psi\to\wt{T}^M_\psi$ denotes the direct limit map. Likewise for $\wt{T}_{\psi(x)}^{M}$ and $\wt{\pi}^M_{\psi(x)}$ for $x\in \RR^M$.
\end{dfn}

\begin{lem}\label{lem:equivalence_coarse_countable_cof}
Let $X\in\HC$ be transitive swo'd with $x_{\cof}\in X$.
Let $M=M^\Gamma_{\ladder}(X)$.
 Let $x\in\RR^{M}$ and $\psi$ be $\exists^\RR\Pi_1$.
 Then the following are equivalent:
 \begin{enumerate}[label=\tu{(}\roman*\tu{)}]
\item\label{item:not_psi} $\SS_{\alphagap}\sats\neg\psi(x)$
 \item\label{item:T-tilde_wfd}  $\wt{T}^{M}_{\psi,x}$
 is wellfounded.
 
\item\label{item:Pi_2-witness_pins_M} There is a $\neg\psi(x)$-prewitness $(N,\Delta)$
such that $N\pins M$ and
 $\rho_\om^N=\om$.
\end{enumerate}
\end{lem}
\begin{proof}
\ref{item:not_psi} $\Rightarrow$ \ref{item:T-tilde_wfd}: This follows easily from Lemma \ref{lem:Sigma_1^alphagap-correct}.

\ref{item:T-tilde_wfd} $\Rightarrow$ \ref{item:Pi_2-witness_pins_M}:
We will actually prove something stronger. For $n<\om$, let $N_n=M|\theta_n^{++}$. We say that $(n,t,u)$ is \emph{relevant} iff $n<\om$ and $(t,u)\in\wt{T}^{N_n}_x$ and $\lh((t,u))=n$. We will show that for each relevant $(n,t,u)$,
\begin{equation}\label{eqn:inductive_hyp}\exists\Delta\in N_n\ \big[(N_n,\Delta)\text{ is a }(\neg\psi(x),t,u)\text{-prewitness}\big].\end{equation}
Applying this with $n=0$ and $(t,u)=(\emptyset,\emptyset)$, we will get some $\Delta\in N_0$ such that $(N_0,\Delta)$ is a $\neg\psi(x)$-prewitness, and then by taking an appropriate hull of $N_0$,
we can get such an $N\pins M|\om_1^{M}$
with $\rho_\om^N=\om$.

For relevant $(n,t,u)$, let $\rank(n,t,u)=\rank^{\wt{T}_x^{M}}(\wt{\pi}^{N_nM}_x(t,u))$.
We will establish line (\ref{eqn:inductive_hyp})  by induction on $\rank(n,t,u)$,
simultaneously for all $n$.

First suppose $(n,t,u)$ is relevant and $\rank(n,t,u))=0$.
 So $(t,u)$ is an end-node of $\wt{T}_x^{N_n}$. Let $\Delta=\{\emptyset\}$. Then  $(N_n,\Delta)$ is a $(\neg\psi(x),t,u)$-prewitness.
 
 Now suppose that we have established line (\ref{eqn:inductive_hyp}) for all relevant $(n,t,u)$ with $\rank(n,t,u)<\alpha$, and fix a relevant
 $(n,t,u)$ with $\rank(n,t,u)=\alpha$. 
 We must show there is a system  $\Delta\in N_n$ such that $(N_n,\Delta)$ is a $(\neg\psi(x),t,u)$-prewitness.
 Given $(t',u')\in \wt{T}_x^{N_n}$ with $\lh((t',u'))=n+1$ and $(t,u)\pins(t',u')$,
 let $(t^*,u^*)=\wt{\pi}^{N_nN_{n+1}}_x((t',u'))$. 
 Then $(t^*,u^*)\in \wt{T}^{N_{n+1}}_x$,  $(n+1,t^*,u^*)$ is relevant and
   $\rank(n+1,t^*,u^*)<\alpha$.
 So by induction, we can fix (the least) $\Delta_{t^*,u^*}\in N_{n+1}$
 such that $(N_{n+1},\Delta_{t^*,u^*})$
 is a $(\neg\varphi(x),t^*,u^*)$-prewitness.
 
 Let \[\bar{N}=\cHull_1^{N_{n+1}}(\theta_n^{+N_n}\cup\{\theta_{n+1}^{N_{n+1}}\}) \]
 and let $\varrho:\bar{N}\to N_{n+1}$ be the uncollapse.
 Then $\bar{N}$ is $1$-sound with $\rho_1^{\bar{N}}=\theta_n^{+N_n}$ and $\pi(p_1^{\bar{N}})=\{\theta_{n+1}^{+N_{n+1}}\}$, and $\varrho$ is $\rSigma_1$-elementary,
 and so $\bar{N}\pins N_n$ by condensation;
 also $\varrho\in M$.
 (The identification of $p_1^{\bar{N}}$
 works because $N_{n+1}|\theta_{n+1}^{+N_{n+1}}\preccurlyeq_1 N_{n+1}$
 and $\theta_{n+1}^{+N_{n+1}}\in\rg(\varrho)$;
 this gives the corresponding $1$-solidity witness in $\rg(\varrho)$.) Note that since $\theta_n^{+N_n}<\crit(\varrho)$, we have $\wt{T}^{N_n}_x\in\bar{N}$
 and $\varrho(\wt{T}^{N_n}_x)=\wt{T}^{N_n}_x$,
 and further, $\wt{T}^{N_{n+1}}_x,\wt{\pi}^{N_nN_{n+1}}_x\in\rg(\varrho)$, and $\Delta_{t^*,u^*}\in\rg(\varrho)$ for each $(t',u')$ and   $(t^*,u^*)$ as in the previous paragraph.
 
 We now specify an appropriate $\Delta$:
  let $\Delta$ be the set of finite sequences $\sigma$ of form
 \[ \sigma=(\sigma_{n+1},\ldots,\sigma_{n+k})\]
 where $k<\om$ and for each $i\in[1,k]$,  $\sigma_{n+i}$ has form
 \[\sigma_i=(N_{n+i},t_{n+i},u_{n+i},\Delta_{n+i}),\]
  and where either $k=0$ (so $\sigma=\emptyset$)
 or $N_{n+1}=\bar{N}$
 and for some $(t',u')\in \wt{T}_x^{N_n}$ with $\lh((t',u'))=n+1$ and $(t,u)\pins(t',u')$,
 we have $\varrho(t_{n+1},u_{n+1},
 \Delta_{n+1})=(t^*,u^*,\Delta_{t^*,u^*})$
 where $t^*,u^*$ are as before,
 and
 \[(\sigma_{n+2},\ldots,\sigma_{n+k})\in \Delta_{n+1}.
 \]
Then it is straightforward to see that $(N_n,\Delta)$ is a $(\neg\varphi(x),t,u)$-prewitness, which completes the induction.

\ref{item:Pi_2-witness_pins_M} $\Rightarrow$ \ref{item:not_psi}:
Fix a $(\neg\psi(x)$-prewitness $N\pins M|\om_1^M$.
Suppose $\Ss_{\alphagap}(\RR)\sats\psi(x)$,
as witnessed by $w\in\RR$.
Since $\lh((\emptyset,\emptyset))=0$,
$N$ is a $0$-partial-$\Gamma$-ladder.
Iterate $N$ at $\theta_0^N$,
making $w$ generic over the image of $P_0(N|\theta_0)$ for the extender algebra at the image of $\theta_0$; let $\Tt_0$
be the tree. 
Then letting $g_0$ be
$\Coll(\om,\theta_0^{M^{\Tt_0}_\infty})$-generic,
there is $(t_0',u_0')\in \wt{T}_x^{M^{\Tt_0}_\infty}[g_0]$
such that $\lh((t_0',u_0'))=1$ and $t_0'=w\rest 1$
and $u_0'(0)$ matches the first norm value for $(x,w)$. Therefore, noting that $(M^{\Tt_0}_\infty,i^{\Tt_0}_\infty(\Delta))$ is a 
$(\neg\psi(x),\emptyset,\emptyset)$-prewitness, 
we can find $\sigma'\in i^{\Tt_0}_{0\infty}(\Delta)$
with $\lh(\sigma')=1$
and $\sigma'=((N',t_0^*,u_0^*,S'))$
for some $N',S'$, with $(t_0^*,u_0^*)=\wt{\pi}^{M^{\Tt_0}_\infty N'}_x((t_0',u_0'))$.
Now drop to $N'$,
and iterate $N'$ above $\theta_0^{+N'}=\theta_0^{+M^{\Tt_0}_\infty}$, making $w$ generic over the image of $P_1(N'|\theta_1^{N'})$.
Then we get some $(t_1',u_1')$
extending $i^{\Tt_1}_{0\infty}(t_0^*,u_0^*)$
with $(t_1',u_1')\in\wt{T}_x^{M^{\Tt_1}_\infty[g_1]}$
and $\lh(t_1',u_1')=2$,
and $t_1'\ins w$, and $u_1'$ giving the first two norm values for $(x,w)$. We can continue in this way for $\om$ steps, and note that this produces a tree with a unique cofinal branch which drops infinitely often (contradiction 1),
and the common part model of the tree is a $\Gamma$-ladder (contradiction 2).\end{proof}

\begin{proof}[Proof of Theorem \ref{tm:anti-correctness_om-standard}]
 Let  $\bar{\alpha}$ and $\pi:\Ss_{\bar{\alpha}}(\RR^M)\to\Ss_{\alphagap}(\RR)$ be the map given by Fact \ref{fact:Sigma_1-elem_pi}.
 
 Let $\varphi$ be $\Pi_2$.
First, we need to specify a $\Sigma_2$ formula $\psi_{\Sigma_2,\varphi}$ (recursively in $\varphi$) such that for all $x\in\RR^M$,
\[ \Ss_{\alphagap}(\RR)\sats\varphi(x)\iff \Ss_{\bar{\alpha}}(\RR^M)\sats\psi_{\Sigma_2,\varphi}(x,x_{\cof}).\]
But by Lemma \ref{lem:equivalence_coarse_countable_cof},
we can just set $\psi_{\Sigma_2,\varphi}$
to be the formula asserting ``there is a $\Pi_1$-iterable $\om$-premouse $R$ over $x_{\cof}$ which is a $\varphi(x)$-prewitness''. By the $\Sigma_1$-elementarity of $\pi$ and since all $\Pi_1$-iterable $\om$-premice $R$ over $x_{\cof}$ with $R\in\HC^M$
are segments of $M$, this works.

Secondly, we need to specify a $\Sigma_2$ formula $\varrho_{\Sigma_2,\varphi}$ such that for all $x\in\RR^M$,
\[ \Ss_{\bar{\alpha}}(\RR^M)\sats\varphi(x)\iff\Ss_{\alphagap}(\RR)\sats\varrho_{\Sigma_2,\varphi}(x,x_{\cof}).\]
For this, first let $\varphi^\RR(x,x_{\cof})$
be the standard $\Pi^\RR_2$ formula
such that \[ \Ss_{\alphagap}(\RR)\sats\all^\RR x\ [\varphi^\RR(x,x_{\cof})\Leftrightarrow\varphi(x)].\]
Now let $\varrho_{\Sigma_2,\varphi}(x,x_{\cof})$ be the formula asserting the existence of a countable $\Pi_1$-iterable premouse $N$ over $x_{\cof}$ such that:
\begin{enumerate}[label=--]\item  $\Lp_{\Gamma}(N|\eta)\ins N$ for every $\eta<\OR^N$, 
\item $N$ is a putative $\Gamma$-ladder,
and
\item $N\sats$``There is $\gamma\in\OR$ such that $\Ss_\gamma(\RR^N)\sats\all^\om n\ \varphi_{\cof}(n,x_{\cof})$,
and letting $\gamma$ be least such,
$\Ss_{\gamma}(\RR^N)\sats\varphi^\RR(x)$''.
\end{enumerate}
Because every $N$ as above has $\RR^M\sub N$, it is straightforward to see that this works.
\end{proof}

\subsection{Antichains of the extender algebra}
\label{sec:antichains}

Work in a premouse $N$
and let $\delta<\OR^N$
be such that $N\sats$``$\delta$ is Woodin''
and $\delta^{+N}<\OR^N$.
Let $\mathscr{L}$
be the collection of formulas $\varphi$ of infinitary propositional logic
in $\om$-many Boolean variables $\{X_n\}_{n<\om}$,
with $\varphi\in V_\delta$.
For $\Gamma\sub\mathscr{L}$
and $\varphi\in\mathscr{L}$,
a \emph{proof of $\varphi$
from $\Gamma$}
is a transitive  $P\sats\ZFC^-$ such that for some transitive swo'd  $X\in P$,
we have $\varphi\in X$ and some
$\Gamma'\in X$ such that $\Gamma'\sub\Gamma$
and $P\sats$ ``$\Coll(\om,X)$ forces that for every real $x$,
if $x\sats\Gamma'$ then $x\sats\varphi$''.
We say that $\Gamma$ is \emph{consistent}
iff for all $\varphi\in\mathscr{L}$,
there is no proof $P$
of $\varphi\wedge\neg\varphi$
from $\Gamma$,
and a formula $\varphi\in\mathscr{L}$ is \emph{consistent with $\Gamma$}
if $\Gamma\cup\{\varphi\}$ is consistent.

Let $\Gamma_0$ be the set of all extender algebra axioms for the extender algebra $\BB_\delta$ at $\delta$ (induced by extenders $E\in\es^{N|\delta}$ with $\nu(E)$ an $N$-cardinal, in the usual way). Clearly $\Gamma_0$ is consistent, since the constantly $0$ real gives models.
As in the proof of the $\delta$-cc for the extender algebra (using the Woodinness of $\delta$),
for any $\varphi\in\mathscr{L}$,
there is a proof $P$
of $\varphi$ from $\Gamma_0$
iff
there is $\Gamma'\in V_\delta$ such that $\Gamma'\sub\Gamma_0$, and  a proof $P'\in V_\delta$ of $\varphi$ from $\Gamma'$. Recall that the extender algebra $\BB_\delta$ consists of all equivalence classes $[\varphi]$ of formulas $\varphi\in \mathscr{L}$
which are consistent with $\Gamma_0$,
modulo the equivalence relation $\approx$, where $\varphi\approx\psi$
iff there is a proof of $\varphi\Leftrightarrow\psi$
from $\Gamma_0$. As  $\delta$ is Woodin,  $\BB_\delta$ has the $\delta$-cc.

Let $\varphi\in\mathscr{L}$.
Then the following are equivalent:
\begin{enumerate}[label=--]\item $[\varphi]\in\BB_\delta$
\item  there is no limit cardinal
$\eta<\delta$
such that $\varphi\in N|\eta$ and $N|\eta\sats$``there is a proof of $\neg\varphi$ from $\Gamma'\cup\{\varphi\}$ from some set $\Gamma'$, with $\Gamma'\sub\Gamma_0^{P|\eta}$'',
where $\Gamma_0^{P|\eta}$ is the set of extender algebra axioms induced by extenders $E\in\es^{N|\eta}$,
\item 
 there is $\eta<\delta$ 
such that $N|\eta\preccurlyeq N|\delta$ and $\varphi\in N|\eta$, and $N|\eta\sats$``there is no such proof''.
\end{enumerate}
So we can uniformly determine whether a given $\varphi\in\mathscr{L}$ yields an element $[\varphi]$ of $\BB_\delta$ in any $\eta<\delta$
with $\varphi\in N|\eta\preccurlyeq N|\delta$.
It follows that the ordering $[\varphi]\leq[\psi]$ is likewise definable, since if $[\varphi],[\psi]\in \BB_\delta$, then  $[\varphi]\leq[\psi]$ iff $[\varphi\wedge\neg\psi]\notin\BB_\delta$. The
 compatibility of $[\varphi]$ with $[\psi]$ is likewise so definable.

Now for $A\sub\mathscr{L}$, let $[A]=\{[\varphi]\bigm|\varphi\in A\}$.
Say that $A$ is \emph{representing} iff for all $\varphi,\psi\in A$,
if $\varphi\neq\psi$
then $[\varphi]\neq[\psi]$.
By the previous paragraph, the questions of whether
$A$ is representing and  whether $[A]$ is an antichain can be similarly determined, over any $N|\eta\preccurlyeq N|\delta$ with $A\in N|\eta$; moreover, if $A$ is representing and $[A]$ is an antichain, then $A\in N|\delta$, by the $\delta$-cc. The question of whether ($A$ is representing and) $[A]$ is a maximal antichain is likewise,
since if $A$ is representing then $[A]$ is maximal
iff $[\neg(\bigvee A)]\notin\BB_\delta$.

\subsection{Q-local local $K^c$-constructions}\label{sec:Q-local_local_K^c}

Let $P$ be an $x$-premouse which has no largest cardinal, where $x\in\RR$. Let $y\leq_Tx$.
Working in $P$, we define the \emph{maximal Q-local-plus-1 certified ms-array over $y$}, by combining features of the local $K^c$-construction of \cite{localKc} with the Q-local $L[\es]$-construction of \cite{mouse_scales}. The main features are as follows.
For simplicity assume $y=0$; the general case is just a relativization (since $y\leq_T x$).

The structure of the construction is basically that of the maximal $+1$ certified ms-array of $P$ (see \cite[p.~5***]{localKc}), except as follows.
We define a \emph{Q-local-plus-1 certificate}
as in the definition of \emph{plus-1
certificate} (\cite[Definition 1.3***]{localKc}) except for the following modifications:
\begin{enumerate}
\item A Q-local-plus-1 certificate
for an ms-array $\left<N_\gamma\right>_{\gamma\leq\eta_0}$ (the ms-array is the sequence of uncored models of the construction) is a triple \[(\left<H_\gamma\right>_{\gamma<\eta_0},\left<F^*_\gamma\right>_{\gamma\leq\eta_0},\left<t_\gamma\right>_{\gamma\leq\eta_0}).\]
\item For each $\gamma\leq\eta_0$, we have $t_\gamma\in\{0,1,3\}$.
\item If $t_\gamma\neq 3$, then $N_\gamma$ is active iff $F^*_\gamma\neq\emptyset$ iff $t_\gamma=1$.
 \item We demand that $\nu(F^*_\gamma)=\aleph_\gamma^{H_\gamma}$, for all $\gamma<\eta_0$
 with $t_\gamma=1$.
 \item We demand that if $t_{\eta_0}=1$
 then $\nu(F^*_{\eta_0})=\aleph_{\eta_0}^{P}$.
 \item We demand that $\OR^{H_{\gamma_0}}<\OR^{H_{\gamma_1}}$ for $\gamma_0<\gamma_1<\eta_0$.
 \item We demand that $\OR^{H_\gamma}<\OR^P$ for $\gamma<\eta_0$.
 \item We set $t_\gamma=3$ iff $N_{\gamma}\sats$``There is a Woodin cardinal''.
 \item Suppose $t_{\gamma_0}=3$,
 and let $\delta$ be the least Woodin cardinal of $N_{\gamma_0}$.
 Then $\OR(N_{\delta})=\delta$
 and $N_{\delta}\ins N_{\gamma_0}$.
 Thus, $t_{\delta+1}=3$. Moreover,
 $P|\delta\sats\ZFC$ and $P|(\delta+\om)\sats$``$\delta$ is the least Woodin''.
 Let $Q$ be the largest segment of $P$ which models ``$\delta$ is Woodin''. Then for $\delta+(\gamma\om)\in(\delta,\OR^Q]$,
 $t_{\delta+\gamma}=3$ and
 $N_{\delta+\gamma}$ is the output of the P-construction of $Q|(\delta+\gamma\om)$ over $N_{\delta}$.
 This makes sense, and
 if $P\sats$``$\delta$ is Woodin'',
 we get $N_{\delta+\gamma}\sats$``$\delta$ is Woodin'' where $\delta+\gamma\om=\OR^P$, whereas if
  $P\sats$``$\delta$ is not Woodin'', we get that $N_{\delta+\gamma}$ is a $\delta$-sound Q-structure for $\delta$ where $\delta+\gamma\om=\OR^Q$ (assuming iterability and sufficient soundness of $P$).
\end{enumerate}
Other than these changes, things are as in \cite[Definition 1.3]{localKc}. We then define the \emph{maximal Q-local-plus-1 certified ms-array} as in \cite[p.~5]{localKc},
but using the above certificates
in place of plus-1 certificates.
Note that this implies that for limit ordinals $\eta$
such that $\kappa_\alpha$ is not eventually constant as $\alpha\to\eta$, if the resulting model $N_{\gamma}$ eventually produced by this construction (so $\OR(N_{\gamma})=\kappa_\eta$
and $N_{\gamma}$ has no largest cardinal) is such that $N_{\gamma}\sats\ZFC$ and $\Jj(N_{\gamma})\sats$``$\delta=\kappa_\eta$ is Woodin'',
then the next stage of the construction (producing segments projecting to $\delta$) has to begin with P-construction, through to either $P$ itself  (if $\delta$ is (the least) Woodin of $P$) or the Q-structure $Q\pins P$ for $\delta$
 (otherwise). One should refer to \cite{mouse_scales} for more details regarding this aspect of the construction.

\subsection{The mouse set theorem for $\OD_{\alpha 2}$}
The analysis in \S\ref{sec:coarse_countable_cof} 
for the $\Ss$-gap $[\alphagap,\alphagap]$ was relative to the real $x_{\cof}$.
We next adapt the  methods
to establish the (lightface) mouse set theorem 
with respect to $\OD_{\alphagap 2}$ and $\OD^\RR_{\alphagap 2}$ (for the same kinds of $\Ss$-gaps as in \S\ref{sec:coarse_countable_cof}).

\begin{dfn}
 Recall that for $\alpha\in\OR$, $n\in[1,\om)$
 and $x\in\RR$,
 $\OD_{\alpha n}(x)$ is the set of reals $y$
 such that for some ordinal $\xi<\om_1$
 and some $\Sigma_n$ formula $\varphi$
 of the $L(\RR)$ language,
 for all $w\in\WO_\xi$
 and all $z\in\RR$, we have
 \[ z=y\iff\Ss_\alpha(\RR)\sats\varphi(x,\xi,z).\]
 And $\OD_{\alpha n}=\OD_{\alpha n}(\emptyset)$.
\end{dfn}

\begin{dfn}
 Recall that 
 $\Sigma_1^\RR=\Sigma_1$,
 $\Pi_1^\RR=\Pi_1$,
 and given $n\geq 1$,
 $\Sigma^\RR_{n+1}=\exists^\RR\Pi_n^\RR$
 and $\Pi^\RR_{n+1}=\neg\all^\RR\Sigma_n^\RR$.
 
 For $\alpha\in\OR$, $n\in[1,\om)$ and $x\in\RR$, $\OD^\RR_{\alpha n}(x)$ is 
 the set of reals $y$
 such that for some ordinal $\xi<\om_1$
 and some $\Sigma^\RR_{n}$ formula $\varphi$
 of the $L(\RR)$ language,
 for all $w\in\WO_\xi$ and
  all $z\in\RR$, we have
 \[ z=y\iff\Ss_\alpha(\RR)\sats\varphi(x,\xi,z).\]
 And $\OD^\RR_{\alpha n}=\OD^\RR_{\alpha n}(\emptyset)$.
\end{dfn}

One should be aware of the potential contrast of the preceding definitions with the next one:
\begin{dfn}
  Given $x,y\in\RR$, we say $y$ \emph{is $\Sigma_n^{\Ss_\alpha(\RR)}(\om_1\cup\{x\})$} 
 iff there is $\xi<\om_1$ and a $\Sigma_n$ formula $\varphi$ such that for all $w\in\WO_\xi$ and all $n<\om$, we have
 \[ n\in y\iff \Ss_\alpha(\RR)\sats\varphi(x,\xi,n).\]
 Likewise define \emph{$y$ is $\Pi_n^{\Ss_\alpha(\RR)}(\om_1)$}.
 And $y$ \emph{is $\Delta_n^{\Ss_\alpha(\RR)}(\om_1)$} iff $y$ is both $\Sigma_n^{\Ss_\alpha(\RR)}(\om_1)$
 and $\Pi_n^{\Ss_\alpha(\RR)}(\om_1)$.
\end{dfn}

\begin{rem}
 Clearly, if $y\in\OD_{\alpha n}$,
 then $y$ is $\Delta_n^{\Ss_\alpha(\RR)}(\om_1)$. If $\alpha$ has uncountable cofinality in $L(\RR)$, then the converse also holds, but
 it doesn't seem that the converse should hold in general.
\end{rem}

 \begin{dfn}
 Say that a real $y$ is $(\alphagap,\RR)$-\emph{cofinal}
 if there is a $\Sigma_1$ formula $\varphi$
 of the $L(\RR)$ language
 and such that $\alphagap$
 is the least ordinal $\alpha$ such that $\Ss_{\alpha}(\RR)\sats\all^\RR x\ \varphi(x,y)$.
\end{dfn}

 \begin{proof}[Proof of Theorem \ref{tm:lightface_mouse_set} for $n=2$]
Fix an $\alphagap$-ascending pair
$(x_{\cof},\varphi_{\cof})$,
 let $\Gamma=\Sigma_1^{\Ss_{\alphagap}(\RR)}$
 and $N=M_{\ladder}^{\Gamma}(x_{\cof})$
 (see \S\ref{sec:coarse_countable_cof}).
 Let $M$ be the output of the Q-local local $K^c$-construction of $N$ \tu{(}over $\emptyset$\tu{)} (see \S\ref{sec:Q-local_local_K^c}).
 We will see that  $M$ witnesses the theorem.

  Let us first establish some general facts about $M$. By \S\ref{sec:Q-local_local_K^c} (and adapting further material from \cite{localKc} and \cite{mouse_scales}),
  for each $n<\om$, since $\rung_n^N$
  is a limit cardinal of $N$,
  $M|\rung_n^N=N_{\rung_n^N}$,
  and since $Q_n^N\sats$``$\rung_n^N$ is Woodin'', $M|\OR^{Q_n^N}$ is the P-construction of $Q_n^N$ over $M|\rung_n^N$.
  Write $Q_n^M=Q_n^N$ (though $Q_n^M$
  need not be defined internally to $M$
  in the manner it is defined in $N$).
  Let $\Sigma^N_n$ be the above-$\rung_n^N$,
  $(\om,\om_1)$-strategy for $Q_n^N$,
  and let $\beta_n^N$ be the least $\beta$
  such that $\Sigma^N_n\in\Ss_{\beta}(\RR)$.
  So $\sup_{n<\om}\beta_n^N=\alphagap$.
  Let $\Sigma^M_n$ be the translation of $\Sigma^N_n$ to an above-$\rung_n^N$,
  $(\om,\om_1)$-strategy
  for $Q_n^M$ (so $\Sigma^M_n$ and $\Sigma^N_n$ are essentially equivalent). Note $M$ has no largest cardinal, and $M\sats$``there is no Woodin cardinal'' (since $N$ has no Woodin, and by the properties of the construction; for example, $Q_n^M$ is the Q-structure for $M|\rung_n^N$).
  
  Now each rung $\theta_n^N$ is a strong cutpoint of $N$, by the $\Sigma_1^N$-definability of the $\Lp_{\Gamma}$-operator.
  It follows that $\theta_n^N$ is  a cutpoint of $M$, and that $\theta_n^N$ is not measurable in $M$ (though it seems that there might be partial measures $E\in\es^M$ with $\crit(E)=\theta_n^N$).
  It follows that for each $m<\om$, $\theta_n^{(+m+1)M}$ is a strong cutpoint of $M$.

  Since $M,N$ each have no largest cardinal, $\rho_1^M=\rho_1^N=\OR^M=\OR^N$.
  Note that this also implies that $0$-maximal trees on these models are equivalent to $1$-maximal trees  (all $\bfrSigma_1^M$ functions $f$ with domain in $M$, are such that $f\in M$; likewise for $N$). So $N$ is $(1,\om_1+1)$-iterable. All iteration trees on $M,N$ we consider will be $1$-maximal.+
  Let $\Psi_M$ be $(1,\om_1+1)$-strategy 
  for $M$ given by lifting/resurrection to $1$-maximal trees on $N$ via $\Sigma_N$.
  Given $\Tt$ on $M$ via $\Psi_M$,
  let $\Uu_\Tt$ be the corresponding tree on $N$, and given $\alpha<\lh(\Tt)$ and $d=\deg^\Tt_\alpha$,
  let $\xi_\alpha^\Tt\leq\OR(M^{\Uu_\Tt}_\alpha)$
  and $\pi_\alpha^\Tt:M^\Tt_\alpha\to \core_d(N_{\xi_\alpha}^{M^{\Uu_\Tt}_\alpha})$ be the standard lifting map. (In particular, $\xi^\Tt_0=\OR^N$, $\deg^\Tt_0=1$
  and $\pi^\Tt_0:M\to \core_1(M)=M$ is the identity.)

  \begin{clm}\label{clm:no_dropping_ladder}
  Let $\Tt$ be a tree on $M$ via $\Psi_M$,
  of successor length. Let $P\ins M^\Tt_\infty$. Then:
  \begin{enumerate}
   \item\label{item:no_Gamma-Woodin} There is no $\delta\leq\OR^P$
   such that $\Lp_{\Gamma}(P|\delta)$ is a premouse and $\Lp_{\Gamma}(P)\sats$``$\delta$ is Woodin''.
   \item\label{item:no_dropping_ladder} The following are equivalent:
   \begin{enumerate}[label=\tu{(}\roman*\tu{)}]\item For all $\alpha<\alphagap$
   there is a $P$-cardinal $\delta<\OR^P$ such that $\Lp_{\Gamma_\alpha}(P|\delta)$ is a premouse and $\Lp_{\Gamma_\alpha}(P|\delta)\sats$``$\delta$ is Woodin'',
   \item $b^\Tt$ is non-dropping, and $P=M^\Tt_\infty$.
   \end{enumerate}
  \end{enumerate}
\end{clm}
\begin{proof}
 Part \ref{item:no_Gamma-Woodin}: Suppose otherwise. In particular, $P|\delta\sats\ZFC$. So if $\delta=\OR^P$ then $P\sats\ZFC$ and $b^\Tt$ does not drop,
 but then $\J(P)\sats$``$\delta$ is not Woodin'', since
 $\OR^M$ is a limit of strong cutpoints of $M$,
 and hence likewise for $\OR^P$ and $P$.
 So $\delta<\OR^P$. Since $M\sats$``there is no Woodin'',
 there is $Q\ins P$ which is a Q-structure for $P|\delta$. We have $\pi^\Tt_\infty:M^\Tt_\infty\to N_{\xi^\Tt_\infty}^{M^{\Uu_\Tt}_\infty}$.
 Let $(\delta',Q')=\pi^\Tt_\infty(\delta,Q)$, where $Q'=N_{\xi^\Tt_\infty}^{M^{\Uu_\Tt}_\infty}$
 if $Q=P$. Then $Q'\pins\Lp_\Gamma(Q'|\delta')$,
 since $Q'$ is just the output of the P-construction of $M^{\Uu_\Tt}_\infty|\OR^{Q'}$
 over $Q'|\delta'$,
 and $M^{\Uu_\Tt}_\infty|\OR^{Q'}\pins\Lp_\Gamma(M^{\Uu_\Tt}_\infty|\delta')$.
 But then using $\pi^\Tt_\infty$,
 we can pull back the above-$\delta'$ strategy
 for $Q'$,
 giving an above-$\delta$ strategy for $Q$
 which is in $\Ss_{\alphagap}(\RR)$,
 so $Q\pins\Lp_\Gamma(Q|\delta)$,
 as desired.
 
 Part \ref{item:no_dropping_ladder}:  This follows from similar considerations as the previous part.
\end{proof}

\begin{clm}
 $\RR\cap M\sub\OD^\RR_{\alphagap 2}$.
\end{clm}
\begin{proof}
Using Claim \ref{clm:no_dropping_ladder}, this is just the direct analogue
of Rudominer's result that $\RR^{M_{\ladder}}\sub\OD^\RR_{\om 2}$. That is,
 let $R\pins M|\om_1^M$ with $\rho_\om^R=\om$. Certainly $R$ is $\Pi_1^{\alphagap}$-iterable,
 so it is enough to see that $R$ is the unique $\om$-premouse $R'$ such that $\OR^{R'}=\OR^R$ and $R$ is $\Pi_1^{\alphagap}$-iterable. But if not, then  comparison of $R$ with $R'$ via the correct (partial) strategies produces a tree on $R$ which violates
  Claim \ref{clm:no_dropping_ladder}.
\end{proof}

  So it just remains to see that $\OD_{\alphagap 2}\sub M$. For this, we consider two cases.

 \begin{case}\label{case:no_alphagap,R-cofinal_real} There is no $(\Tt,R,\lambda,g,y)$ such that $\Tt$ is a successor length tree  on $M$ via $\Psi_M$, $b^\Tt$ is non-dropping, $R=M^\Tt_\infty$,
   $\lambda<\OR^R$, $g$  is $(R,\Coll(\om,\lambda))$-generic,
and $y\in R[g]$ is an
$(\alphagap,\RR)$-cofinal real.

In this case, we
will show 
$\OD_{\alphagap 2}=
\OD_{\alphagap 1}=\OD^{<\alphagap}=\RR\cap M$.

\begin{clm}\label{clm:R|theta^+=Lp(R|theta)}
Let $\Tt$ be a successor length tree on $M$ via $\Psi_M$ such that $b^\Tt$ does not drop. Let $R=M^\Tt_\infty$. Then
$R|\theta^{+R}=\Lp_{\Gamma}(R|\theta)$
for each strong cutpoint $\theta$ of $R$.
\end{clm}
\begin{proof}
We have $\Lp_\Gamma(R|\theta)\ins R$,
since otherwise letting $P\pins\Lp_\Gamma(R|\theta)$ be such that $\rho_\om^P\leq\theta$
and $P\npins R$, we can compare $R$ with $P$,
producing trees $\Tt',\Uu'$ respectively,
with $\Tt'$ being $1$-maximal, $M^{\Tt'}_\infty\ins M^{\Uu'}_\infty$ and $b^{\Tt'}$ does not drop in model or degree. Now there is $\alpha<\alphagap$ such that $P$ is above-$\theta$ iterable in $\Ss_\alpha(\RR)$. And there is $n<\om$ such that $\theta<i^{\Tt}_{0\infty}(\rung_n^N)$ and $Q_n^M$ is not above $\rung_n^N$-iterable in $\Ss_\alpha(\RR)$. But $i^{\Tt'}_{0\infty}(i^\Tt_{0\infty}(Q_n^M))$ is above $i^{\Tt'}_{0\infty}(i^\Tt_{0\infty}(\rung_n^M))$-iterable in $\Ss_\alpha(\RR)$, and pulling back this iteration strategy with $i^{\Tt'}_{0\infty}\com i^\Tt_{0\infty}$ gives a contradiction.

Now suppose for a contradiction that $\Lp_\Gamma(R|\theta)\pins R|\theta^{+R}$; and suppose for the moment that $\Tt$ is trivial, so $R=M$.
Let $P\pins M$ be  such that $\Lp_\Gamma(M|\theta)=M|\theta^{+M}$ and $\rho_\om^M\leq\theta$. Let $g$ be $(M,\Coll(\om,\theta))$-generic
and $y\in M[g]$ be a real coding $P$.
Now for all sound premice $P'$
such that $\theta$ is a strong cutpoint of $P'$, $P'|\theta^{+P'}=P|\theta^{+P}$, and $\rho_\om^{P'}\leq\theta$ but $P'\neq P$,
we have that $P'$ is not above-$\theta$, $\Pi_1^{\Ss_{\alphagap}}(\RR)$-iterable.
(Otherwise comparing $P$ with $P'$
gives an above-$\theta$ tree $\Tt$ on $P$ which generates a $\Gamma$-ladder, contradicting Claim \ref{clm:no_dropping_ladder}.) Note that this statement is $\all^\RR\Sigma_1(\{P,\theta\})$,
and is true in $\Ss_{\alphagap}(\RR)$. So, since there is no $(\alphagap,\RR)$-cofinal real in $M[g]$,
there is $\alpha<\alphagap$
which already satisfies the statement.
But $P$ is above-$\theta$, $\Pi_1$-iterable
in $\Ss_{\alphagap}$, and hence also in $\Ss_\alpha$. So $P$ is the unique sound premouse $P'$ such that $\theta$ is a strong cutpoint of $P'$, $P'|\theta^{+P'}=P|\theta^{+P}$,
$\rho_\om^{P'}\leq\theta$, and $P$ is $\Pi_1$-iterable in $\Ss_\alpha(\RR)$. Therefore $P\in\OD^{\alpha}(P|\theta)$, so $P\in\OD^{<\alphagap}(P|\theta)\sub\Lp_{\Gamma}(P|\theta)=P|\theta^{+P}$, a contradiction.

Now suppose $\Tt$ is non-trivial.
First consider the case that $\theta=i^\Tt_{0\infty}(\bar{\theta})$
for some strong cutpoint $\bar{\theta}$ of $M$.
Note that $i^\Tt_{0\infty}$ is continuous at
$\bar{\theta}^{+M}$. So fix $\bar{P}\pins M$
such that $\rho_\om^{\bar{P}}\leq\bar{\theta}\leq\OR^{\bar{P}}$; it suffices to see that $P=i^\Tt_{0\infty}(\bar{P})\pins\Lp_\Gamma(R|\theta)$. We can fix $n<\om$ and some $W\ins Q_n^N$ with $\OR^{\bar{P}}<\rung_n^N<\OR^W$ and
such that $W$ is an above-$\rung_n^N$, $\varphi(\bar{P},\bar{\theta})$-prewitness where $\varphi(P',\theta')$ asserts ``there is an above-$\theta'$ iteration strategy for $P'$''. Since $b^\Tt$ is non-dropping, $b^{\Uu_\Tt}$ is also non-dropping, and $\pi^\Tt_\infty:M^\Tt_\infty\to N^{M^{\Uu_\Tt}_\infty}_{\xi_\infty}$
and $\pi^\Tt_\infty\com i^{\Tt}_{0\infty}=i^{\Uu_\Tt}_{0\infty}\rest M$.
So $i^{\Uu_\Tt}_{0\infty}(W)\ins Q_n^{M^{\Uu_\Tt}_\infty}$ and $i^{\Uu_\Tt}_{0\infty}$ is an above-$\rung_n^{M^{\Uu_\Tt}_\infty}$, $\varphi(\pi_{\infty}^{\Tt}(P),\pi_{\infty}^\Tt(\theta))$-prewitness. But $Q_n^{\Uu_\Tt}$ is above-$\rung_n^{\Uu_\Tt}$ iterable in $\Ss_{\alphagap}(\RR)$,
so $\pi_\infty^\Tt(P)$ is above-$\pi_\infty^\Tt(\theta)$ iterable in $\Ss_{\alphagap}(\RR)$,
so $P$ is above-$\theta$ iterable in $\Ss_{\alphagap}(\RR)$, which suffices.

Now fix a strong cutpoint $\theta_0$  of $M$.
Note that for every strong cutpoint $\theta$ of $M$ with $\theta<\theta_0$, and for every $P\pins M$ with $\theta\leq\OR^P$ and $\rho_\om^P\leq\theta$,
there is an above-$\theta_0$, $\varphi(P,\theta)$-prewitness $W$ such that $W\pins M|\theta_0^{+M}$. This statement is preserved by $i^{\Tt}_{0\infty}$,
so for every strong cutpoint $\theta$ of $R$ with $\theta<i^\Tt_{0\infty}(\theta_0)$,
and for every $P\pins R$ with $\theta\leq\OR^P$
and $\rho_\om^P\leq\theta$,
there is an above-$i^\Tt_{0\infty}(\theta_0)$,
$\varphi(P,\theta)$-prewitness $W$ such that $W\pins R|i^{\Tt}_{0\infty}(\theta_0)^{+R}$.
But we have already seen that $R|i^{\Tt}_{0\infty}(\theta_0)^{+R}=\Lp_\Gamma(R|\theta_0)$, and it follows that $P$ is above-$\theta$
iterable in $\Ss_{\alphagap}(\RR)$, as desired.
\end{proof}

 By Claim \ref{clm:R|theta^+=Lp(R|theta)},
  $M|\om_1^M=\Lp_{\Gamma}(\emptyset)$,  so $\RR^M=\OD_{\alphagap 1}=\OD_{<\alphagap}$.
 
 Now let $y\in\OD_{\alphagap 2}$. It suffices to see that $y\in M$.
 Let $\varphi$ be a $\Pi_1$ formula and $\eta<\om_1$  be such that for all $w\in\WO_\eta$ and for all $z\in\RR$, we have
 \[ z=y\iff\Ss_{\alphagap}(\RR)\sats\exists X\ \varphi(X,w,z).\]
 Let $\Tt$ be the length $\eta+1$ linear iteration of $M$ at its least measurable.
 Let $R=M^\Tt_\infty$. Let $\alpha_0$ be least such that there is $X\in\Ss_{\alpha_0}(\RR)$ and $w\in\WO_\eta$ such that
  $\Ss_{\alphagap}(\RR)\sats\varphi(X,w,z)$.
  Let $n<\om$ be such that $Q_n^M$ is not above-$\rung_n^N$ iterable in $\Ss_{\alpha_0}(\RR)$.
  Clearly $\eta<\OR(i^{\Tt}_{0\infty}(Q_n^M))$.
  Let $\theta$ be a strong cutpoint of $M$
  with $\OR(Q_n^M)<\theta$. Let $W\pins M|\theta^{+M}$ be an above-$\theta$, $\psi(Q_n^M,\rung_n^N)$-prewitness, where $\psi(Q',\theta')$ asserts ``there is an above $\theta'$ iteration strategy for $Q'$''.
  
  We claim that for $k<\om$,
  we have $k\notin y$ iff $R|i^\Tt_{0\infty}(\theta)^{+R}\sats$``there is  $U\pins L[\es]$
  which is an above-$i^\Tt_{0\infty}(\theta)$,
  $\tau(k,\eta,i^\Tt_{0\infty}(Q),i^\Tt_{0\infty}(\rung_n^N))$-prewitness,
  where $\tau(k',\eta',Q',\theta')$ asserts ``$\exists\alpha,\beta$ such that $\alpha\leq\beta$ and $\Ss_\alpha(\RR)\sats$``$Q'$ is above-$\theta'$ iterable, and I am the least level of $L(\RR)$ with such an iteration strategy'',
  and $\Ss_\beta(\RR)\sats$``letting $\pi:\RR\to\Ss_\alpha(\RR)$ be the resulting canonical surjection, we have
  \[ \all^\RR x,w,z\ \big[k\in z\wedge w\in \WO_\eta\implies\ \neg\varphi(\pi(x),w,z)\big]\text{''}\text{''}.\]
  This is straightforward to verify, using (i) the
  fact that $R|i^\Tt_{0\infty}(\theta)^{+R}=\Lp_\Gamma(R|i^\Tt_{0\infty}(\theta))$ and $\Lp_\Gamma(R|i^\Tt_{0\infty}(\theta))$ has segments which yield $\varrho$-witnesses for all the relevant $\Sigma_1^{L(\RR)}$ assertions,
  and (ii) letting $\alpha_0'$
  be the least $\alpha'$ such that $i^\Tt_{0\infty}(Q)$ is above-$i^\Tt_{0\infty}(\theta_n^N)$ iterable
  in $\Ss_{\alpha'}(\RR)$, and $\pi:\RR\to\Ss_{\alpha_0'}(\RR)$ the resulting canonical surjection, then if $k\in\om\cut y$, we have
  \[ \Ss_{\alphagap}(\RR)\sats\all^\RR x,w,z\ \big[k\in z\wedge w\in\WO_\eta\implies\neg\varphi(\pi(x),w,z)\big],\]
  and so by case hypothesis, there is $\beta\in[\alpha_0',\alphagap)$ such that $\Ss_\beta(\RR)$ satisfies the same statement.

  This shows that $y$ is definable over $R|i^\Tt_{0\infty}(\theta)^{+R}$,  so $y\in R$, so $y\in M$, as desired.
  \end{case}
  
  \begin{case}\label{case:alphagap,R-cofinal_real} Otherwise;
  that is, there is $(\Tt,R,\lambda,g,y)$
  as described in Case \ref{case:no_alphagap,R-cofinal_real}.
  
  In this case, we will just show $\OD_{\alphagap 2}=\RR\cap M$.
  
  \begin{scase}\label{scase:M_has_alphagap,R-cofinal_real} There is an $(\alphagap,\RR)$-cofinal real $y_0\in M$.

 Easily, $\Sigma_2^{\Ss_{\alphagap}(\RR)}$
 reduces recursively to $\exists^\RR\Pi_1^{\Ss_{\alphagap}(\RR)}(\{y_0\})$.  We will show that $\exists^\RR\Pi_1^{\Ss_{\alphagap}(\RR)}$
 can be appropriately defined over $M$,
 by doing a version of what we did in \S\ref{sec:coarse_countable_cof}.
 
Fix a $\Pi_1$ formula $\psi(u,v,w)$. Fix $x\in\RR^M$.
 Consider the game $\mathscr{G}=\mathscr{G}_{\exists^\RR\psi(x,y_0)}$ in which player 2 attempts to build  $(X,\left<g_n\right>_{n<\om},w)$
  such that  $X\preccurlyeq_1 M$,   $\{\rung_n^N\}_{n<\om}\sub X$, $g_n$
  is $(P_n^N,\BB_{\rung_n^N}^{Q_n^N})$-generic,
  $w$ is the corresponding generic real (independent of $n$), $g_n=g_{n+1}\cap\BB_{\theta_n^N}^{Q_n^N}$, and
  there is no $n<\om$ and $W\pins Q_n^N$
  such that $W[w]$ is an above-$\rung_n^N$, $\neg\psi(x,y_0,w)$-prewitness.
  
  In more detail, this is as follows. Fix a recursive enumeration $\left<\psi_i\right>_{i<\om}$ of all formulas in the passive premouse language.
  Write $\BB_i=\BB_{\rung_i^N}^{Q_i^M}$ (the extender algebra of $Q_i^M$ at $\rung_i^N$).
  The game  $\mathscr{G}$ has length $\om$. In round $n<\om$, player 1 first plays some $\vec{a}_n\in (M|\rung_n^N)^{<\om}$, and then player 2 plays
  $\vec{b}_n,\vec{x}_n\in (M|\rung_n^N)^{<\om}$
  such that:
  \begin{enumerate}
  \item (Cofinality of $X$) $\vec{a}_n\cup\vec{b}_n\sub\vec{x}_n$,
   and if $n>0$ then $\rung_{n-1}^N\in\vec{x}_n$.
   \item ($\Sigma_1$-elementarity of $X$): If $n>0$ then for all $i,j,k<n$, letting $\vec{x}'_\ell=\vec{x}_\ell\cap M|\rung_i^N$,  if \[ M|\rung_i^N\sats\exists z\ \psi_j(\vec{x}'_0,\ldots,\vec{x}'_{k-1},z) \]
   then there is some $x\in\vec{x}_n$ such that $M|\rung_i^N\sats\psi_j(\vec{x}'_0,\ldots,\vec{x}_{k-1}',x)$.
      \item $\vec{b}_n=\left<b_{ni}\right>_{i\leq n}$
   where $b_{ni}\in\BB_i$
   for each $i\leq n$.
   \item If $n>0$ then $b_{ni}\leq^{\BB_i} b_{n-1,i}$ for each $i<n$.
   \item For all $i\leq j\leq n$, (noting that $b_{ni}\in\BB_j$) we have $b_{nj}\leq^{\BB_j} b_{ni}$.
   \item If $n>0$ then for each $A\in\vec{x}_{n-1}$ and each $i< n$,
   if $A$ is a maximal antichain of $\BB_{i}$
   then there is $a\in A$ such that $b_{ni}\leq^{\BB_i} a$.
   
   \item\label{item:avoid_counterexample} For each $i\leq n$, there is no $W\pins Q_i^M$ with $\rung_i^N<\OR^W$ such that \[ Q_i^M\sats b_{ni}\forces_{\BB_i}\text{``}W\text{ is an above-}\rung_i^N\text{, }\neg\psi(x,y_0,\dot{w})\text{-prewitness}\text{''},\]
   where $\dot{w}$ denotes the $\BB_i$-generic real.
  \end{enumerate}

  Now  $\mathscr{G}$ is closed for player 2,
  so one of the players has a winning strategy.
  We claim that player 2 has a winning strategy iff $\Ss_{\alphagap}(\RR)\sats\exists^\RR w\ \psi(x,y,w)$.
  
  First suppose player 2 wins  $\mathscr{G}$,
  and let $\sigma$ be a winning strategy.
  Then we can clearly play against $\sigma$, ensuring that $X=\bigcup_{n<\om}\vec{x}_n=M$.
  This yields $\left<g_n\right>_{n<\om}$
  such that for each $n<\om$, $g_n$ is $(Q_n^M,\BB_n)$-generic, and $g_n\sub g_{n+1}$.
  So letting $w_n$ be the generic real given by $g_n$, in fact $w_n$ is independent of $n$,
  so write $w=w_n$. Then $\Ss_{\alphagap}(\RR)\sats\psi(x,y,w)$, since the cofinality of the $Q_n^M$'s in $\Ss_{\alphagap}(\RR)$ ensure that all potential counterexamples to $\psi$ have been ruled out (using condition \ref{item:avoid_counterexample} in the rules of $\mathscr{G}$).
  
  Now suppose instead that player 1 wins $\mathscr{G}$. Then the game tree associated to $\mathscr{G}$ is ranked in the usual manner. That is, define sets $W_\alpha$
  of positions in $\mathscr{G}$, which will be winning positions for player 1, by recursion on $\alpha$. We will have $W_\alpha\sub W_\beta$ for $\alpha<\beta$, and $\emptyset\in W_\infty$. Start with $W_0=$ the set of finite partial plays
  \begin{equation}\label{eqn:partial_play}\sigma=\left<\vec{a}_i,\vec{b}_i,\vec{x}_i\right>_{i<n},\end{equation}
  with player 1 next to move, in which player 1 has already won. Then given $\left<W_\alpha\right>_{\alpha<\beta}$,
  let $W_\beta$ be the set of all finite partial plays $\sigma$ as in (\ref{eqn:partial_play})
  such that there is $\vec{a}_n\in(M|\rung_n^N)^{<\om}$
  such that for all $\vec{b}_n,\vec{x}_n\in(M|\rung_n^N)^{<\om}$, we have $\sigma\conc\left<\vec{a}_n,\vec{b}_n,\vec{x}_n\right>\in\bigcup_{\alpha<\beta}W_\alpha$.
  Then $W_\alpha\sub W_\beta$ for all $\alpha<\beta$. Let $\alpha_\infty$
  be least such that $W_{\alpha_\infty+1}=W_{\alpha_\infty}$.
  Then for any finite partial play $\sigma$
  as above, player 1 has a winning strategy from $\sigma$ iff $\sigma\in W_{\alpha_\infty}$.
  In particular, $\emptyset\in W_{\alpha_\infty}$.
  
  We now define a \emph{$\Pi_2$-witness}
  by analogy with the corresponding notion in \S\ref{sec:coarse_countable_cof}.

\begin{dfn}\label{dfn:fine_Pi_2-witness}
Given a transitive swo'd $X\in\HC$, an $X$-premouse $N$ 
and $n<\om$, we say that $(N,\rung_0,\ldots,\rung_n)$ is  \emph{$n$-partial-potential-ladder}
iff:
\begin{enumerate}[label=--]\item $\theta_0<\theta_1<\ldots<\theta_n$,
\item $\theta_i$ is a  cutpoint of $N$, for each $i\leq n$,
\item $N\sats$``$\rung_i$ is a limit cardinal which is \emph{not} Woodin and \emph{not} measurable''
for each $i\leq n$, and
\item $\rung_n^{+N}<\OR^N$ and $\rung_n^{+N}$ is the largest cardinal of $N$.
\end{enumerate}

Suppose $(N,\vec{\theta})$ is an $n$-partial-potential-ladder. Then let $Q_i^{(N,\vec{\theta})}\ins N$ denote the Q-structure for $N|\rung_i$, for $i\leq n$.
Given $x,y_0\in\RR^N$ and a $\Pi_1$ formula $\psi(u,v,w)$ of the $L(\RR)$ language,
let $\mathscr{G}^{(N,\vec{\theta})}=\mathscr{G}_{\exists^{\RR}\psi(x,y_0)}^{(N,\vec{\theta})}$
denote the set of partial plays
$\sigma=\left<(\vec{a}_i,\vec{b}_i,\vec{x}_i)\right>_{i<n}$
(of length exactly $n$) according to the rules for $\mathscr{G}_{\exists^\RR\psi(x,y_0)}$ described earlier, relative to $\left<(\theta_i,Q_i)\right>_{i<n}$
where $Q_i=Q_i^{(N,\vec{\theta})}$.
\end{dfn}

\begin{dfn}Let $X\in\HC$ be transitive swo'd,
 $N$ be an $X$-premouse,  $\vec{\theta},\varrho\in N^{<\om}$, $x,y_0\in\RR^N$, and $\Delta\in N$.
Let $\psi(u,v,w)$ be a $\Pi_1$ formula of the $L(\RR)$ language. We say that $(N,\vec{\theta},\Delta)$ is a \emph{$(\all^\RR \neg\psi(x,y_0),\varrho)$-prewitness}
iff, letting $n=\lh(\varrho)$, then $(N,\vec{\theta})$ is an $n$-partial-potential-ladder (so $\lh(\vec{\theta})=n+1$),
and letting $(N_n,\vec{\theta}_n,\varrho_n,\Delta_n)=(N,\vec{\theta},\varrho,\Delta)$,  then
 $\varrho\in\mathscr{G}_{\exists^\RR\psi(x,y_0)}^{(N,\vec{\theta})}$
 and $\Delta$
 is a non-empty tree (set of finite sequences  closed under initial segment) whose elements $\sigma$ have form
\[ \sigma=(\sigma_{n+1},\ldots,\sigma_{n+k}) \]
where $k<\om$ and for $0\leq i\leq k$, $\sigma_{n+i}$ has form
\[ \sigma_{n+i}=(N_{n+i},\vec{\theta}_{n+i},\varrho_{n+i},\Delta_{n+i}) \]
(so $\sigma_n=(N,\vec{\theta},\varrho,\Delta)$, but $\sigma_n$ is not actually an element of $\sigma$), and moreover, for each $\sigma\in \Delta$, with $\sigma,k,\sigma_{n+i}$ as above, the following conditions hold:
\begin{enumerate}
\item If $\sigma\neq\emptyset$ then for every $i<k$, we have the following:
\begin{enumerate}\item
$N_{n+i+1}\pins N_{n+i}$,
\item $(N_{n+i+1},\vec{\theta}_{n+i+1})$ is an $(n+i+1)$-potential-partial-ladder,
\item $\rho_1^{N_{n+i+1}}=\rho_\om^{N_{n+i+1}}=\theta_{n+i}^{+N_{n+i}}$,
\item 
$\vec{\theta}_{n+i+1}\rest(n+i+1)=\vec{\theta}_{n+i}$,
\item $\varrho_{n+i+1}\in \mathscr{G}_{\exists^\RR\psi(x,y_0)}^{(N_{n+i+1},\vec{\theta}_{n+i+1})}$ (so $\lh(\varrho_{n+i+1})=n+i+1$),
\item $\varrho_{n+i}=\varrho_{n+i+1}\rest(n+i)$,
\item $\Delta_{n+i+1}\in N_{n+i+1}$
is a  tree,
\item 
$ (\sigma_{n+i+2},\ldots,\sigma_{n+k})\in \Delta_{n+i+1}$.
\end{enumerate}
\item
$\Delta_{n+k}=\{\tau\bigm|\sigma\conc\tau\in \Delta\}$.
\item 
Letting $\theta_{n+k}=\max(\vec{\theta}_{n+k})$,
there is $\vec{a}\in (N_{n+k}|\rung_{n+k})^{<\om}$ such that for all $\vec{b},\vec{x}\in(N_{n+k}|\rung_{n+k})^{<\om}$ such that \[\varrho'=\varrho_{n+k}\conc\left<(\vec{a},\vec{b},\vec{x})\right>\in\mathscr{G}_{\exists^\RR\psi(x,y_0)}^{(N_{n+k},\vec{\theta}_{n+k})},\]
there is $\sigma'\in\Delta$
such that $\sigma'=\sigma\conc\left<(N',\vec{\theta}',\varrho',\Delta')\right>$
for some $N',\vec{\theta}',\Delta'$.
\end{enumerate}

A \emph{\emph{$(\all^\RR \neg\psi(x,y_0))$-prewitness}}
is a $(\all^\RR \neg\psi(x,y_0),\emptyset)$-prewitness.
\end{dfn}

\begin{lem}
 There is a $(\all^\RR\neg\psi(x,y_0))$-prewitness $(N,\vec{\theta},\Delta)$
 with $N\pins M|\om_1^M$.
\end{lem}
\begin{proof}
This follows from the following more general fact: For each $n<\om$
and  finite partial play $\varrho$
of $\mathscr{G}_{\exists^\RR\psi(x,y_0)}$
with $\varrho\in W_{\alpha_\infty}$,
there is a $((\all^\RR\neg\psi(x,y_0)),\varrho)$-prewitness $(N,\vec{\theta},\Delta)$
such that $N\pins M|\rung_n^{++M}$.
(Applying this to $n=0$ and $\varrho=\emptyset$, we get some $(N',\vec{\theta}',\Delta')$ with $N'\pins M|\rung_0^{++M}$ (and $\lh(\vec{\theta}')=1$).
But then by condensation, we get some such $(N,\vec{\theta},\Delta)$ with $N'\pins M|\om_1^M$.) The proof of this is by induction on ranks of nodes in $W_{\alpha_\infty}$,
much like in the proof of Lemma \ref{lem:equivalence_coarse_countable_cof}, so we omit further detail.\footnote{Recall that $\rung_n^N$
is a strong cutpoint of $N$. It follows that $\rung_n^N$ is a non-measurable cutpoint in $M$.
It seems it might be that $\rung_n^N$ fails to be a strong cutpoint of $M$, however,
since coring might lead to partial measures $E\in\es^M$ with $\crit(E)=\rung_n^N$. Hence the formulation of Definition \ref{dfn:fine_Pi_2-witness} in this regard.}
\end{proof}
\begin{lem}
 Suppose there is a $(\all^\RR\neg\psi(x,y_0))$-prewitness $(N,\vec{\theta},\Delta)$
 with $N\pins M|\om_1^M$.
 Then $\Ss_{\alphagap}(\RR)\sats\all^\RR w\ \neg\psi(x,y_0,w)$.
\end{lem}
\begin{proof}
Suppose otherwise and let $w\in\RR$ be a counterexample. Iterate $N$ to $N'$, making $w$ generic for the image of the extender algebra at $\theta_0$ where $\vec{\theta}=\left<\theta_0\right>$. Let $\Tt_0$
be the iteration tree doing this.
Let $g_0$ be the $(Q'_0,\BB_{\theta'_0}^{Q'_0})$-generic determined by $w$, where $Q'_0\pins N'$ is the Q-structure for $\theta'_0=i^{\Tt}_{0\infty}(\theta)$. Let $a_0\in (N'|\theta'_0)^{<\om}$ witness the choice of $(N,\vec{\theta},\Delta)$ (which is preserved by $i^\Tt_{0\infty}$). Choose $\vec{b}_0,\vec{x}_0$ such that $(\vec{a}_0,\vec{b}_0,\vec{x}_0)\in\mathscr{G}^{(N',\vec{\theta}'_0)}_{\exists^\RR\psi(x,y_0)}$
with $\vec{b}_0$ consistent with $g_0$.
This is possible  since there is no conflict with condition \ref{item:avoid_counterexample}
of the rules of $\mathscr{G}_x$,
since $\Ss_{\alphagap}(\RR)\sats\psi(x,y_0,w)$,
and $Q'_0=i^\Tt_{0\infty}(Q_0)$
is above-$\theta'_0$ iterable in $\Ss_{\alphagap}(\RR)$ (so any $W$ as in condition \ref{item:avoid_counterexample} would also have to be above-$\theta'_0$ iterable in $\Ss_{\alphagap}(\RR)$, but this would imply that $\Ss_{\alphagap}(\RR)\sats\neg\psi(x,y_0,w)$). Then we can find $\sigma\in \Delta'=i^\Tt_{0\infty}(\Delta)$
with $\lh(\sigma)=1$,
and $\sigma$ consistent with $(\vec{a}_0,\vec{b}_0,\vec{x}_0)$.
Let $\sigma=(N_1,\vec{\theta}_1,\varrho_1,\Delta_1)$. Now iterate $N_1$ to make $w$ generic for the image of the extender algebra of $Q_1\pins N_1$ at $\theta_1$ (where $Q_1$ is the Q-structure for $N_1|\theta_1$,
and where $\vec{\theta}_1=(\theta',\theta_1)$).
Since we already made $w$ generic over $N'_0$,
and $\theta'_0$ is a cutpoint of $N_1$,
this iteration is above $\theta'_0$.
We can carry on in this manner throughout all finite stages $n$. But this produces a correct tree which drops in model infinitely often along its unique branch, a contradiction.
\end{proof}
  \end{scase}
\begin{scase}
 There is no $(\alphagap,\RR)$-cofinal real $y\in M$.
 
 Fix $(\Tt,R,\lambda,g,y)$
  witnessing  Case \ref{case:alphagap,R-cofinal_real} hypothesis;
  this tuple is as described in Case \ref{case:no_alphagap,R-cofinal_real}. Note we may assume that $\lambda$ is a strong cutpoint of $R$.
  The argument of Subcase \ref{scase:M_has_alphagap,R-cofinal_real} relativizes above $(R|\lambda,g)$, so letting $x\in\RR$ be the natural real coding $(R|\lambda,g)$, we have
    $\OD_{\alphagap 2}(x)\sub R[g]$. Therefore $\OD_{\alphagap 2}\sub R[g]$. But this is independent of $g$, and so $\OD_{\alphagap 2}\sub R$, so $\OD_{\alphagap 2}\sub M$, as desired.
\end{scase}
\end{case}
\end{proof}

\subsection{The mouse set theorem for $\OD_{\alpha,n+3}$}

Recall that we have already defined $\Pi_1^{\alphagap}$-iterability.
\begin{dfn}
For $n\geq 1$ and $k\leq\om$, we define 
 \emph{$\Pi^{\alphagap}_{n+1}$-$(k,\om_1)$-iterability} for $X$-premice $P\in\HC$, where $X\in\HC$ is transitive swo'd. The definition is by recursion on $n$, simultaneously for all $k$ and $X$.
 Say that a $k$-sound $X$-premouse $P\in\HC$ is \emph{$\Pi_{n+1}^{\alphagap}$-$(k,\om_1)$-iterable}
 iff for every  putative $k$-maximal
 tree $\Tt$ on $P$ of length $\eta<\om_1$, for every ordinal $\lambda\in\Lim\cup\{0\}$ with $\lambda\leq\eta$,
 and every sequence $\left<Q_\xi\right>_{\xi\in\Lim\cap(0,\lambda]}$
 such that for every limit ordinal $\xi\leq\lambda$, we have:
 \begin{enumerate}[label=--]\item  $Q_\xi$ is an above-$\delta(\Tt\rest\xi)$, $\Pi_n^{\alphagap}$-$(\om,\om_1)$-iterable $\delta(\Tt\rest\xi)$-sound
 Q-structure for $M(\Tt\rest\xi)$, and
 \item if  $\xi<\lambda$ then $Q_\xi\ins M^\Tt_\xi$,
 \end{enumerate}
then:
 \begin{enumerate}[label=--]
  \item if $\lambda<\lh(\Tt)$ and either $\lambda=0$ or $Q_\lambda\ins M^\Tt_\lambda$
  then $\Tt\rest(\min(\eta,\lambda+\om))$ has wellfounded models, and
  \item there is a $\Tt\rest\lambda$-cofinal branch $b$ such that $Q_\lambda\ins M^{\Tt\rest\lambda}_b$.\qedhere
 \end{enumerate}
\end{dfn}

Note that $\Pi^{\alphagap}_{n+1}$-iterability
is $(\Pi^{\RR}_{n+1})^{\Ss_{\alphagap}(\RR)}$.

\begin{dfn}Let $\alphagap$ be $\om$-standard,
 $(x_{\cof},\varphi_{\cof})$  be $\alphagap$-ascending, $\Gamma=\Sigma_1^{\Ss_{\alphagap}(\RR)}$. 
 Let $X\in\HC$ be transitive swo'd 
 with $x_{\cof}\in X$. Let
 $N$ be an  $X$-premouse.
 For $n<\om$,
 we say that $N$ is \emph{$n$-$\Gamma$-ladder-large} iff there is  $N'\ins N$
 and $\delta_0,\ldots,\delta_{n-1}\in(\rank(X),\OR^{N'})$ such that:
 \begin{enumerate}[label=--]\item $\delta_0<\ldots<\delta_{n-1}$,
  \item 
for each $i<n$,
 $N'\sats$``$\delta_i$ is Woodin'', and
 \item $N'$ is a putative $\Gamma$-ladder as an $N|\delta$-premouse,
 where $\delta=\delta_{n-1}$ if $n>0$, and $\delta=0$ otherwise.
 \end{enumerate}
 We say that $N$ is \emph{$n$-$\Gamma$-ladder-small}
 iff it is not $n$-$\Gamma$-ladder-large.
 
 For $y\in\RR$, let $M_{n,\ladder}^\Gamma(x_{\cof},y)$ 
 denote the minimal sound $1$-$\Gamma$-large $(x_{\cof},y)$-mouse. (So $M_{0,\ladder}^\Gamma(x_{\cof},y)=M_{\ladder}^\Gamma(x_{\cof},y)$.)\end{dfn}
 
\begin{proof}[Proof of Theorem \ref{tm:lightface_mouse_set} for $n>2$]
 
Let $m<\om$ and
 $N=M_{m+1,\ladder}^\Gamma(x_{\cof})$. Let $M$ be the output of the Q-local $L[\es]$-construction of $N$ (as in \cite{mouse_scales}).\footnote{We don't use the Q-local local $K^c$-construction here, just because it's not necessary. It would also work fine though.} We will show that $\OD_{\alphagap, m+3}=\OD^\RR_{\alphagap,m+3}=\RR\cap M$.

 Let $\Psi_M$ be the $(1,\om_1+1)$-strategy
 for $M$ given by lifting/resurrection to $\Sigma_N$. Given $\Tt$ on $M$ via $\Psi_M$,
 let $\Uu=\Uu_\Tt$ be the corresponding tree on $N$,
 and for $\alpha<\lh(\Tt)$, let $\xi_\alpha\leq\OR^{M^\Uu_\alpha}$
 and $\pi_\alpha:M^\Tt_\alpha\to \core_d(N^{M^\Uu_\alpha}_{\xi_\alpha})$ be the standard lifting map, where $d=\deg^\Tt_\alpha$.

 \begin{clm}\label{clm:no_dropping_m+1-ladder}
  Let $\Tt$ on $M$ be a tree via $\Psi_M$ of countable successor length $\eta+1$, and
  let $R\ins M^\Tt_\eta$.
  Then the following are equivalent:
  \begin{enumerate}[label=\tu{(}\roman*\tu{)}]
   \item $(0,\eta]^\Tt\cap\dropset^\Tt=\emptyset$ and $R=M^\Tt_\eta$,
   \item $(0,\eta]^\Tt\cap\dropset_{\deg}^\Tt=\emptyset$ and $R=M^\Tt_\eta$,
   \item\label{item:m+1-ladder} there are $\delta_0<\ldots<\delta_{m}<\OR^R$ such that:
   \begin{enumerate}[label=--]\item $R\sats$``$\delta_i$ is Woodin'' for all $i\leq m$,  and
   \item for all $\alpha<\alphagap$,
   there is $\rung\in(\delta_{m},\OR^R)$
   such that $\rung$ is a $R$-cardinal
   and there is $Q\pins R$ such that $\rung<\OR^Q$
   and $Q\sats$``$\rung$ is Woodin''
   and there is no above-$\rung$ iteration strategy for $Q$ in $\Ss_\alpha(\RR)$.
   \end{enumerate}
  \end{enumerate}
 \end{clm}
\begin{proof}
 Since trees $\Uu$ via $\Sigma_N$ have the corresponding properties (since they are appropriately first-order relative to $x_{\cof}$), this follows in a straightforward manner from the existence of the lifting maps $\pi_\alpha$
 and properties of the Q-local $L[\es]$-construction.
\end{proof}

 \begin{clm}
  $\RR\cap M\sub\OD^\RR_{\alphagap,m+3}$.
 \end{clm}
 \begin{proof}
 It is enough to see that for each $P\pins M$ with $\rho_\om^P=\om$, letting $\xi=\OR^P$,
 $P$ is the unique $\Pi_{n+2}^{\alphagap}$-iterable $\om$-premouse such that $\OR^{P'}=\xi$. But this follows from Claim \ref{clm:no_dropping_m+1-ladder} via a quite routine comparison argument.
 (Compare $P$ with another candidate $P'$.
 Since $P\pins M|\om_1^M$,
 the tree $\Tt$ on $P$ drops, for no $\eta<\lh(\Tt)$
 is there $R\ins M^\Tt_\eta$  and $\delta_0<\ldots<\delta_m<\OR^R$ as in clause \ref{item:m+1-ladder} of Claim \ref{clm:no_dropping_m+1-ladder}. Because of this, the $\Pi^{\alphagap}_{n+2}$-iterability of $P'$ is enough to complete the comparison.)
\end{proof}

\begin{clm}\label{clm:OD_alhphagap,m+3_sub_M}
 $\OD_{\alphagap,m+3}\sub M$.
\end{clm}
\begin{proof}
Suppose first that $m=0$, so $N=M_{1,\ladder}^\Gamma(x_{\cof})$. Let $y\in\OD_{\alphagap,3}$.
Let $\xi<\om_1$ and $\varphi$ be a $\Sigma_3$ formula such that for all $w\in\WO_\xi$ and all $z\in\RR$, we have
\[ z=y\iff\Ss_{\alphagap}(\RR)\sats\varphi(z,w).\]
Let $\varphi'$ be the natural $\Sigma^\RR_3$
formula such that for all $x\in\RR$, if $(x,\varphi_{\cof})$ is $\alphagap$-ascending then 
\[ \Ss_{\alphagap}\sats\all^\RR z,w\ \big[\varphi(z,w)\iff\varphi'(z,w,x)\big].\]
Let $\psi'$ be $\Pi^\RR_2$, such that
\begin{equation}\label{eqn:varphi'_iff_exists^RR_t_psi'} \varphi'(z,w,x)\iff\exists^\RR t\ \psi'(t,z,w,x).\end{equation}

Now $\delta_0^N$ is the least Woodin cardinal of $M$ (so write $\delta_0^M=\delta_0^N$), $M|\delta_0^M$ is definable over $N|\delta_0^N$, $N|\delta_0^N$ is extender algebra generic over $M$, and $M$ is the P-construction of $N$ above $M|\delta_0^M$. Since $\Lp_{\Gamma}(N|\delta_0^N)\pins\Lp_{\Gamma}^{N|\delta_0^{+N}}$
and $(\delta_0^M)^{+M}=(\delta_0^N)^{+N}$,
$\Lp_{\Gamma}(M|\delta_0^M)$ is just the P-construction of $\Lp_{\Gamma}(N|\delta_0^N)$ over $M|\delta_0^M$, and $\Lp_{\Gamma}(M|\delta_0^M)\pins M|(\delta_0^M)^{+M}$. Let $\zeta=\OR(\Lp_\Gamma(M|\delta_0^M))=\OR(\Lp_\Gamma(N|\delta_0^N))$.
Note $\cof^N(\zeta)=\om$
(since $x_{\cof}\in N$).
Let $\mu=\cof^M(\zeta)$.
Then $\mu$ is not measurable in $M$,
since every measurable cardinal of $M$
is also measurable in $N$. So letting $\Tt$
be any countable length tree on $M$ via $\Psi_M$
of length $\eta+1$,
if $(0,\eta]^\Tt\cap\dropset^\Tt=\emptyset$ (so $(0,\eta]^\Tt\cap\dropset_{\deg}^\Tt=\emptyset$),
and $\Uu=\Uu_\Tt$,
then $i^\Tt_{0\eta}$ and $i^\Uu_{0\eta}$ are both continuous at $\zeta$,
and since $\pi_\eta\com i^\Tt_{0\eta}=i^\Uu_{0\eta}\rest M$, it easily follows that \[i^\Tt_{0\eta}(\Lp_\Gamma(M|\delta_0^M))=\Lp_\Gamma(M^\Tt_\eta|i^\Tt_{0\eta}(\delta_0^M)).\]

Let $\Tt$ be the length $(\xi+1)$ linear iteration on $M$ at its least measurable, and $M'=M^\Tt_\xi$.
Let $\delta'=i^\Tt_{0\xi}(\delta)$.
So $\delta'$ is the least Woodin of $M^\Tt_\xi$ and $i^\Tt_{0\xi}(\Lp_\Gamma(M|\delta_0^M))=\Lp_\Gamma(M^\Tt_\xi|\delta')$.
Let $\eta'=\OR(\Lp_\Gamma(M^\Tt_\xi|\delta'))$.

Let $n<\om$.
It is enough to see that for each $\ell<\om$, we have:

\begin{sclm}\label{sclm:compute_y_in_M'} $\ell\in y$
iff $M'|(\delta')^{+M'}\sats$``there is $p\in \BB_{\delta',>\xi}$   which forces that, letting $\dot{g}$ be the standard name for the generic filter and $(\dot{t},\dot{z},\dot{w},\dot{x})$ the standard name for the generic real,
we have:
\begin{enumerate}\item $\dot{w}\in\WO_{\check{\xi}}$,
 \item For every $k<\om$ there is an above-$\delta'$, $\varphi_{\cof}(k,\dot{x})$-prewitness $J$ such that $J\pins (M'|\eta')[\dot{g}]$.
 \item For every $k<\om$ there is an above-$\delta'$, $\varphi^*_{\cof}(k,\dot{x})$-prewitness $J$ such that $J\pins (M'|\eta')[\dot{g}]$,
 where $\varphi^*_{\cof}(u,v)$ asserts ``there is an ordinal $\gamma$ such that $\Ss_\gamma(\RR)\sats\varphi_{\cof}(a,b)\wedge\neg\varphi_{\cof}(u+1,v)$''.
 \item There is no above-$\delta'$,
 $\psi(\dot{x})$-prewitness $J$ such that $J\pins (M'|\eta')[\dot{g}]$, where $\psi(u)$ asserts ``there is an ordinal $\gamma$ such that $\Ss_\gamma(\RR)\sats\all^\om k\ \varphi_{\cof}(k,u)$''.
 \item\label{item:Pi_2-prewitness_m=0} There is an above-$\delta'$, $\psi'(\dot{t},\dot{z},\dot{w},\dot{x})$-prewitness $J$,
 as defined in \ref{dfn:varphi(x)-prewitness_tree_version}, but relative to the putatively $\alphagap$-cofinal pair $(\dot{x},\varphi_{\cof})$, such that $J\pins L[\es,\dot{g}]$.\footnote{Here $\psi'$ was the $\Pi_2^\RR$ formula from line (\ref{eqn:varphi'_iff_exists^RR_t_psi'}). Recall that we are presently working in $M'|(\delta')^{+M'}$, so $L[\es]$ denotes this model. And the phrase ``relative to the putatively $\alphagap$-cofinal pair'' means that we use $(\dot{x},\varphi_{\cof})$ in exactly the manner in which we earlier used $(x_{\cof},\varphi_{\cof})$ in Definition \ref{dfn:varphi(x)-prewitness_tree_version}.}
 \item $\ell\in \dot{z}$.''
\end{enumerate}
\end{sclm}

\begin{proof}
Suppose first that $M'|(\delta')^{+M'}$ satisfies
the indicated statement. Let $g$ be $(M',\BB_{\delta',>\xi})$-generic with $p\in g$.
Let $(t,z,w,x)$ be the generic real. It is enough to see that $z=y$, and for this,  it is enough to see that  $w\in\WO_\xi$, $(x,\varphi_{\cof})$ is $\alphagap$-ascending and $\Ss_{\alphagap}(\RR)\sats\psi'(t,z,w,x)$. But considering what $p$ forces, certainly $w\in\WO_\xi$,
and because $M'|\eta'=\Lp_{\Gamma}(M'|\delta')$,
also $(x,\varphi_{\cof})$ is $\alphagap$-ascending.
Finally, we have $J\pins M'|(\delta')^{+M'}[g]$
which is an above-$\delta'$, $\psi'(t,z,w,x)$-prewitness (in the sense of \ref{dfn:varphi(x)-prewitness_tree_version}) relative to $(x,\varphi_{\cof})$. Since $(x,\varphi_{\cof})$ is indeed $\alphagap$-ascending,
and therefore Lemma \ref{lem:equivalence_coarse_countable_cof}
holds with respect to $(x,\varphi_{\cof})$ (replacing  $(x_{\cof},\varphi_{\cof})$), 
it is enough to see that \[M'[g]=M^\Gamma_{\ld}((M'|\delta')[g]),\] with the $M^\Gamma_{\ld}$-operator also defined relative to $(x,\varphi_{\cof})$.
But this follows easily enough from the facts that:
\begin{enumerate}[label=--]\item $\Lp_\Gamma(M'|\delta')[g]=\Lp_\Gamma((M'|\delta')[g])$,
\item  $\alphagap$ is the least ordinal $\alpha$ such that every proper segment of $\Lp_\Gamma((M'|\delta')[g])$ is above-$\delta'$ iterable in $\Ss_\alpha(\RR)$,
and
\item $\eta'=\OR(\Lp_\Gamma((M'|\delta')[g])$
is the least ordinal $\eta''$ such that for each $k<\om$,
there is an above-$\delta'$, $\varphi_{\cof}(k,x)$-prewitness $J$ with $(J\pins M'|\eta'')[g]$,
\item $M'[g]$ is sound as an $(M'|\delta',g)$-premouse.
\end{enumerate}

Now suppose that $\ell\in y$. Then because we can iterate $M'$, above $\xi+1$, to make $(t,y,w,x_{\cof})$ generic for the image of $\BB_{\delta',>\xi}$,
and $(\delta')^{+M'}<\OR^{M'}$, and by the calculations used in the previous direction,
it is easy to see that $M'|(\delta')^{+M'}$ satisfies the desired statement, completing the proof of the subclaim, and hence that of Claim \ref{clm:OD_alhphagap,m+3_sub_M} in case $m=0$.
\end{proof}

Now suppose $m=1$, so $N=M^\Gamma_{2,\ladder}(x_{\cof})$. We proceed as before, but ``$3$'' is replaced by ``$4$'',
and in particular, $\varphi$ is $\Sigma_4$, $\varphi'$ is $\Sigma^\RR_4$ and $\psi'$ is $\Pi^\RR_3$. Let $\tau'$ be $\Sigma^\RR_2$,
such that
\begin{equation}\label{eqn:psi'_and_tau'}\psi'(t,z,w,x)\iff\all^\RR s\ \tau'(s,t,z,w,x).\end{equation}

Again $\delta_0^M=\delta_0^N$ is the least Woodin of $M$,
etc, and $M$ is the P-construction of $N$ above $M|\delta_0^M$ (not just above $\delta_1^N$).
We have $\xi$ as before, and iterate $M$ to $M'$
at the least measurable of $M$ as before, with tree $\Tt$.
Let $\delta_i'=i^\Tt_{0\xi}(\delta_i^M)$,
for $i=0,1$. Let $\eta_i'=\OR(\Lp_\Gamma(M^\Tt_\xi|\delta_i'))$.

\begin{sclm}\label{sclm:compute_y_in_M'_m=1}
 $\ell\in y$ iff $M'|(\delta_1')^{+M'}\sats$``there is $p\in\BB_{\delta_0',>\xi}$ which forces that,
 letting $\dot{g}_0$ be the standard name for the generic filter and $(\dot{t},\dot{z},\dot{w},\dot{x})$ the standard name for the generic real, 
 we have:
 \begin{enumerate}
  \item $\dot{w}\in\WO_{\check{\xi}}$,
 \item For every $k<\om$ there is an above-$\delta_0'$, $\varphi_{\cof}(k,\dot{x})$-prewitness $J$ such that $J\pins (M'|\eta_0')[\dot{g}_0]$.
 \item For every $k<\om$ there is an above-$\delta_0'$, $\varphi^*_{\cof}(k,\dot{x})$-prewitness $J$ such that $J\pins (M'|\eta_0')[\dot{g}_0]$,
 where $\varphi^*_{\cof}(u,v)$ asserts ``there is an ordinal $\gamma$ such that $\Ss_\gamma(\RR)\sats\varphi_{\cof}(a,b)\wedge\neg\varphi_{\cof}(u+1,v)$''.
 \item There is no above-$\delta_0'$,
 $\psi(\dot{x})$-prewitness $J$ such that $J\pins (M'|\eta'_0)[\dot{g}_0]$, where $\psi(u)$ asserts ``there is an ordinal $\gamma$ such that $\Ss_\gamma(\RR)\sats\all^\om k\ \varphi_{\cof}(k,u)$'',
 \end{enumerate}
 and $\BB_{\delta_1',>\delta_0'}$ forces that,
 letting $\dot{g}_1$ be the standard name for the generic filter and $\dot{s}$ that for the generic real, we have:
 \begin{enumerate}[label=\tu{(}\alph*\tu{)}]
 \item\label{item:no_Pi_2-prewitness_m=1} There is \underline{no} above-$\delta'_1$, $\neg\tau'(\dot{s},\dot{t},\dot{z},\dot{w},\dot{x})$-prewitness $J$,
 as defined in \ref{dfn:varphi(x)-prewitness_tree_version}, but relative to $(\dot{x},\varphi_{\cof})$, such that $J\pins L[\es,\dot{g}_0,\dot{g}_1]$.\footnote{Recall $\tau'$
 is the $\Sigma_2^\RR$ formula in line (\ref{eqn:psi'_and_tau'}), so $\neg\tau'$ is $\Pi^\RR_2$.}
 \item $\ell\in \dot{z}$.''
 \end{enumerate}

\end{sclm}

The subclaim is proved much like Subclaim \ref{sclm:compute_y_in_M'}, and completes the proof of Claim \ref{clm:OD_alhphagap,m+3_sub_M} in case $m=1$.
(Note that aside from the extra Woodin,
a key difference between the cases of $m=0$ and $m=1$
arises between clause \ref{item:Pi_2-prewitness_m=0} of Subclaim \ref{sclm:compute_y_in_M'} and clause \ref{item:no_Pi_2-prewitness_m=1} of Subclaim \ref{sclm:compute_y_in_M'_m=1}.)

The cases $m>1$ are likewise, alternating in details with parity. This completes the proof of Claim \ref{clm:OD_alhphagap,m+3_sub_M}.
\end{proof}
This completes the proof of Theorem \ref{tm:lightface_mouse_set} for $n>2$.
\end{proof}
\section{Admissible gaps}\label{sec:admissible_gaps}

In this section, we will adapt arguments of previous sections to admissible gaps $[\alphagap,\betagap]$.
That is, $\Ss_{\alphagap}(\RR)$ is admissible.
So $[\alphagap,\betagap]$ is either a weak or strong gap. If it is weak,
let $\wh{\betagap}=\betagap$, and if strong, let $\wh{\betagap}=\betagap+\om$. 
We write $\Ss=\Ss_{\wh{\betagap}}(\RR)$.

We make use of the $\Pg$-operator associated to $\Ss$, defined in \cite{gaps_as_derived_models}.
Let $a\in\RR$  be such that the $\Pg$-operator $X\mapsto\Pg(X)$ is defined for transitive swo'd $X\in\HC$ with $a\in X$, with $P=\Pg(X)$ a sound $\om$-small $X$-mouse projecting to $X$ with $\Lp_{\Gammagap}(X)=P|\Theta_{X^{<\om}}^P$\footnote{Here $\Theta_{X^{<\om}}^P$ is the supremum of all ordinals which are the surjective image of $X^{<\om}$ in $P$. Actually we could restrict our attention to self-wellordered $X$, in which case $\Theta_{X^{<\om}}^P=\rank(X)^{+P}$.} and $P$ has $\om$-many Woodin cardinals $>\rank(X)$ in ordertype (so $P=L_\xi[P|\lambda^{P}]$
where $\lambda^{P}$ is the sup of those Woodins and $\xi\in\OR$). Let $x_{\cone}\in\RR$ be such that $\Pg(a)\leq_T x_{\cone}$. Let $\mathscr{C}$
be the transitive swo'd $\HC$-cone above $x_{\cone}$; that is,
the set of all transitive swo'd $X\in\HC$ with $x_{\cone}\in X$.  Also as in \cite{gaps_as_derived_models},
let $\beta^*$ be the end of the gap in the $\M^{\alphagap}$-hierarchy.

Let's suppose for simplicity that $\rho_1^{\Pg(x_{\cone})}=\om<\lambda<\OR^{\Pg(x_{\cone})}$
and $\lambda^{\Pg(x_{\cone})}\notin p_1^{\Pg(x_{\cone})}$, and letting $\beta^*$ be as in \cite{gaps_as_derived_models}, $\mSigma_1^{\M_{\beta^*}}$ is $\mu$-reflecting and $\beta^*$
is closed enough that $\beta^*=\wh{\betagap}$.
So for all  $X\in\mathscr{C}$, we have $\lh(p_1^{\Pg(X)})=\lh(p_1^{\Ss})$.
And by \cite{gaps_as_derived_models}, we can fix a $\Sigma_1$ formula $\psi_{\Sat}$ such that for all $X\in\mathscr{C}$,  all $\rSigma_1$ formulas $\varphi$
and all $\vec{x}\in X^{<\om}$, we have
\begin{equation}\label{eqn:psi_Sat} \Pg(X)\sats\varphi(\vec{x},p_1^{\Pg(X)})\iff\Ss\sats\psi_{\Sat}(\varphi,\vec{x},X,p_1^{\Ss}).\end{equation}

Let $t_a$ be the theory $\Th_{\rSigma_1}^{\Pg(a)}(\{a,p_1^{\Pg(a)}\})$.

\begin{lem}
There is an $\rSigma_1$ formula $\varphi_0$
such that for all $X\in\mathscr{C}$, letting $P=\Pg(X)$, we have:
\begin{enumerate}\item\label{item:Pg(X)_is_least_satisfying_all_n_varphi_0} $P$ is the least $P'\ins P$ with $p_1^{P}\in P'$ and 
$P'\sats\all^\om n\ \varphi_0(n,t_a,p_1^{\Pg(X)})$,
and 
\item\label{item:varphi_0_is_str_inc} letting $\alpha_n^{P}$ be the least $\alpha$ such that $\max(p_1^{P})<\alpha$ and  $P|\alpha\sats\varphi_0(n,t_a,p_1^{P})$,
then $\alpha_n^{P}<\alpha_{n+1}^{P}$ for all $n<\om$.\footnote{Here if $P=\Jj(P|\beta)$ for some $\beta$,
then we allow $\alpha_n^{P}$ to be a successor ordinal, with $P|\alpha_n^N$ being defined via Jensen's $\Ss$-hierarchy.}
\end{enumerate}
\end{lem}
\begin{proof}
Part \ref{item:Pg(X)_is_least_satisfying_all_n_varphi_0}:
Recall that $\Pg(a)\leq_T x_{\cone}$,
so $\Pg(a),t_a\in \Jj(X)$. Let $N$ be the output of the Q-local
 $L[\es,a]$-construction of $\Pg(X)$ (see \cite{mouse_scales}). Then $N$ is a correct non-dropping iterate of $\Pg(a)$, $\delta_0^{N}=\delta_0^{\Pg(X)}$, $N$ is the P-construction of $P$ over $N|\delta_0^N$, $N$ is $\delta_0^N$-sound, and $p_1^{N}=p_1^{P}$.
 So for all $\rSigma_1$ formulas $\varphi$, we have
 \[ \Pg(a)\sats\varphi(a,p_1^{\Pg(a)}) \iff N\sats\varphi(a,p_1^{P}).\]
 But these $\Sigma_1$ truths are verified cofinally in $\OR^N=\OR^{P}$,
 so by taking $\varphi_0(n,t,p_1^{P})$
 to assert that $N$ satisfies the $n$th statement in $t[p_1^{\Pg(a)}/p_1^{P}]$ (this denotes the theory resulting from $t$ by replacing $p_1^{\Pg(a)}$ with $p_1^{P}$), $\varphi_0$ is as desired.
 
 Part \ref{item:varphi_0_is_str_inc}: This is obtained via a simple modification of the formula just constructed.
 \end{proof}
 
 We fix a formula $\varphi_0$ witnessing the lemma.

 \begin{dfn}\label{dfn:T_n^N}
Let $X\in\mathscr{C}$.
An $X$-premouse $N$ is called \emph{relevant}
if there is $\delta<\OR^N$
such that $N\sats$``$\delta$ is the least Woodin $>\rank(X)$'' (and write $\delta^N=\delta$),
 $N=\Pg(N|\delta)$,
 and $\Pg(N|\beta)\pins N$ for each $\beta<\delta$.

Let $N$ be relevant. Then $\alpha_n^N$
denotes the least $\alpha$
such that $\max(p_1^N)<\alpha<\OR^N$ and $N|\alpha\sats\varphi_0(n,t,p_1^{N})$.
 Define
\[ \gamma^N_n=\sup(\delta^N\cap\Hull_1^{N|\alpha_n^N}(\{p_1^N\})),\]
and
\[ T_n^N=\Th_{1}^{N|\alpha_n^N}(\gamma_n^N\cup\{p_1^N\}).\]
Note $\gamma^N_n$ is a limit cardinal of $N$ and $\gamma^N_n<\delta^N$.\end{dfn}

Now also by \cite{gaps_as_derived_models}, we can fix a $\Sigma_1$ formula $\psi_{\Th}$
such that for all $X\in\mathscr{C}$, all $t\in\Lp_{\Gammagap}(X)$, and all $n<\om$, letting $P=\Pg(X)$, we have:
\begin{equation}\label{eqn:psi_Th} t=\Th_{\Sigma_1}^{P|\alpha_n^{P}}(\{p_1^{P}\}\cup X) 
 \iff \Ss\sats\psi_{\Th}(x_{\cone},n,t,X,p_1^{\Ss}).
\end{equation}

\begin{dfn}
Let $X\in\mathscr{C}$ and $N$ be an $X$-premouse.
We say  that $N$ is a \emph{$\Pgap$-ladder} iff $N$ is relevant
and for each $n<\om$,
there is a relevant $P\pins N$
such that $\gamma_n^N<\OR^P$ and $T_n^P=T_n^N$ (after identifying $p_1^P$ with $p_1^N$; it follows that $\gamma_n^P=\gamma_n^N$). Write $P_n^N=P$ for the least such $P$. We write
$M_{\ladder}^{\Pg}(X)$ for the least $X$-mouse which is a  $\Pgap$-ladder.\end{dfn}
 
  For each $X\in\mathscr{C}$, $M_{\ladder}^{\Pg}(X)$ exists (assuming $\AD^{L(\RR)}$): for example,  the least relevant $X$-mouse
 $N$ such that $\rho_1^N=\delta^N$
 is a $\Pgap$-ladder.
 
 \begin{lem}
 Let $M=M^{\Pg}_{\ld}(X)$.
 Then $\delta^M=\sup_{n<\om}\gamma_n^M$.\end{lem}
 \begin{proof}Suppose not and let $\bar{\delta}=\sup_{n<\om}\gamma_n^M$ and $\bar{M}=\Hull_1^M(\bar{\delta}\cup\{p_1^{M}\}\cup X)$,
 and $\pi:\bar{M}\to M$ the uncollapse.
 Then $\crit(\pi)=\bar{\delta}$
 and $\pi(\bar{\delta})=\delta$,
 and  by calculations in \cite{gaps_as_derived_models}, $\bar{M}$ is also a $\Pgap$-ladder, and since $\bar{\delta}<\delta^M$
 and $M$ is $\Pgap$-closed,
 we have $\bar{M}\pins M$,
 contradicting the minimality of $M$.
 \end{proof}

 \begin{dfn}
 Let $X\in\mathscr{C}$ and  $N$ be an  $X$-premouse.
 
 We say $N$ is \emph{$\Pgap$-correct} iff for each $\xi<\OR^{N}$ with $\rank(X)\leq\xi$,
 either $\Pgap(N|\xi)\ins N$ or $N\pins\Lp_{\Gammagap}(N|\xi)$ and $N$ is $\xi$-sound and projects $\leq\xi$. And $N$ is called \emph{$\Pgap$-closed} iff for all $\xi<\OR^N$ with $\rank(X)\leq\xi$, we have $\Pgap(N|\xi)\pins N$.
 
 Suppose $N$ is $\Pgap$-correct
 and $n$-sound. Let $\Tt$ be an $n$-maximal tree on $N$. We say that $\Tt$ is \emph{$\Pgap$-transcendent} iff for each $\alpha+1<\lh(\Tt)$,
 if $M^\Tt_\beta$ is $\Pgap$-correct for all $\beta\leq\alpha$ then
  $M^\Tt_\alpha||\lh(E^\Tt_\alpha)$ is $\Pgap$-closed.

  Let $\Tt$ be an $n$-maximal tree on an $n$-sound $X$-premouse $Q$.  We say that $\Tt$ is \emph{$\Gammagap$-guided}
  iff for every limit $\eta<\lh(\Tt)$,
  there is $Q\ins M^\Tt_\eta$
  which is a Q-structure for $M(\Tt\rest\eta)$,
  and $Q\pins\Lp_{\Gammagap}(M(\Tt\rest\eta))$.
  If $\Tt$ has limit length, we say that $\Tt$ is \emph{$\Gammagap$-short}
  iff $\Tt$ is $\Gammagap$-guided and there is $Q\pins\Lp_{\Gammagap}(M(\Tt))$ which is a Q-structure for $M(\Tt)$, and we say that $\Tt$ is \emph{$\Gammagap$-maximal} otherwise (that is, if $\Lp_{\Gammagap}(M(\Tt))$ is a premouse and satisfies ``$\delta(\Tt)$ is Woodin''.
\end{dfn}

\begin{dfn}
Let $X\in\mathscr{C}$ and $N$ be an $n$-sound $X$-premouse.

 Say that an $N$ is \emph{$\Pi_1^{\Ss}(\{p_1^{\Ss},x_{\cone}\})$-$n$-iterable} (abbreviated \emph{$\Pi_1$-$n$-iterable})
 iff $N$ is $\Pgap$-correct and 
 for every putative $n$-maximal iteration tree $\Tt$ on $N$ which is $\Gammagap$-guided and $\Pgap$-transcendent,  we have:
 \begin{enumerate}
  \item $\Tt$ has wellfounded $\Pgap$-correct models,
  \item if $\Tt$ has limit length and is $\Gammagap$-short and $Q\pins\Lp_{\Gammagap}(M(\Tt))$ is the Q-structure for $M(\Tt)$,
  then there is a $\Tt$-cofinal branch $b$ with $Q\ins M^\Tt_b$, and
  \item\label{item:Pi_1-it_maximal_clause} if $\Tt$
 has limit length and is $\Gammagap$-maximal
 and $M(\Tt)\sats$``there is no Woodin cardinal $>\rank(X)$''
 then for every $n<\om$
 there is a $\Tt$-cofinal branch $b$
 and some $Q\ins M^\Tt_b$ such that:
 \begin{enumerate}[label=--]\item 
 $\Lp_{\Gammagap}(M(\Tt))=Q|\delta(\Tt)^{+Q}$,
 \item $Q$ has $\om$ Woodins $>\delta(\Tt)$,
 \item  $\rho_1^Q\leq\delta(\Tt)<\lambda^Q<\OR^Q$ and $p_1^Q\cut(\lambda^Q+1)=\emptyset$,
 \item defining $\alpha_k^Q$, $\gamma_k^Q$ and
  $T_k^Q$ as in \ref{dfn:T_n^N} (we do not assume that $Q=\Pgap(M(\Tt))$, but can still define these objects), we have $\sup_{k<\om}\alpha_k^Q=\OR^Q$ and $\sup_{k<\om}\gamma_k^Q=\delta(\Tt)$, and moreover,
  $T_n^Q=T_n^{\Pgap(M(\Tt))}$ (for the particular $n$ under consideration).\qedhere
  \end{enumerate}
 \end{enumerate}
\end{dfn}
 
\begin{lem}
 Let $X\in\mathscr{C}$ and $n\leq\om$.
The set of  $\Pi_1^\Ss(\{p_1^\Ss,x_{\cone}\})$-$n$-iterable $X$-premice is $\Pi_1^{\Ss}(\{p_1^{\Ss},x_{\cone}\})$,
uniformly in $X,n$.\end{lem}
\begin{proof}

First note that $\HC$ and wellfoundedness for elements of $\HC$ are  $\Delta_1^{\Ss}(\{p_1^{\Ss},x_{\cone}\})$,
as is $\{\alphagap\}$, and using these we can express everything other than $\Pg$-correctness and the equation $T_n^Q=T_n^{\Pg(M(\Tt))}$ in part \ref{item:Pi_1-it_maximal_clause}. (For $\Pg$-transcendence, given $\Pg$-correctness, we can just say that for each $\alpha+1<\lh(\Tt)$ and each $\xi<\lh(E^\Tt_\alpha)$,
we have $\Lp_{\Gammagap}(M^\Tt_\alpha|\xi)\pins M^\Tt_\alpha|\lh(E^\Tt_\alpha)$.  And note that the quantifier ``there is a $\Tt$-cofinal branch''
is not a problem.)

Now $N$ is $\Pg$-correct iff for every $\xi<\OR^N$,
either:
\begin{enumerate}[label=(\roman*)]
 \item\label{item:N_is_Gamma_short_at_top} $N\pins\Lp_{\Gammagap}(N|\xi)$ and $N$ is $\xi$-sound and projects $\leq\xi$, or
 \item\label{item:Pg(N|xi)_seg_of_N} $\Pgap(N|\xi)\ins N$.
\end{enumerate}
Clause \ref{item:N_is_Gamma_short_at_top} is easily expressible, as already discussed.
To assert clause \ref{item:Pg(N|xi)_seg_of_N},
we use the definability of $\Lp_{\Gammagap}$
and standard calculations like those used in connection the proof that solidity of the standard parameter ensures its preservation under iteration maps. That is, we assert that, letting $\eta\leq\OR^N$ be such that $\Lp_{\Gammagap}(N|\xi)=N||\eta$, then in fact $\Lp_{\Gammagap}(N|\xi)\pins N$, and there is (a uniquely specified) $Q\ins N$ such that $\eta=\xi^{+Q}<\OR^Q$ and $\rho_1^Q\leq\xi$ and $\xi$ is a strong cutpoint of $Q$, and for each $\rSigma_1$ formula $\varphi$ and each $\vec{x}\in(N|\xi)^{<\om}$, 
\[ \Ss\sats\psi(x_{\cone},\varphi,\vec{x},N|\xi,p_1^\Ss) \implies Q'\sats\varphi(\vec{x},p_1^{Q'}) \]
where $Q'$ is the reorganization of $Q$ as a $N|\xi$-premouse (and recall that the language for $N|\xi$-premice has a constant symbol interpreted as $N|\xi$, so $\varphi$ can refer to $N|\xi$).
The key is now that these properties ensure that $Q'=\Pg(N|\xi)$, since letting $P=\Pg(N|\xi)$, certainly
\[ \Th_{\rSigma_1}^{P}(\{p_1^{P}\}\cup(N|\xi)^{<\om})\ins^*\Th_{\rSigma_1}^{Q}(\{p_1^Q\}\cup(N|\xi)^{<\om}),\]
where $\ins^*$ means both $\sub$ and that the theory on the left is an initial segment of that on the right with respect to the usual prewellorder on $\rSigma_1$ truth. But then if the theories were not equal, we would get $P\pins Q$,
but $\eta=\xi^{+P}$ and $\rho_1^{P}\leq\xi<\eta$, contradicting that $\eta=\xi^{+Q}$.

We can similarly assert that later models  $M^\Tt_\alpha$ of $\Tt$ are $\Pg$-correct (and so our earlier expression of $\Pg$-transcendence will be valid).

Finally, to assert the equation ``$T_n^Q=T_n^{\Pg(M(\Tt))}$''
in the context of condition \ref{item:Pi_1-it_maximal_clause}, we just have to be able to identify $T_n^{\Pg(M(\Tt))}$ in a simple enough manner. But this is achieved via the formula $\Psi_{\Th}$ and line (\ref{eqn:psi_Th}).
 \end{proof}
 
 $\Pi_1$-iterability is enough to determine the proper segments of $M_{\ld}^{\Pg}(X)$, but not enough  for $M_{\ld}^{\Pg}(X)$ itself.
 
 \begin{lem}
 Let $X\in\mathscr{C}$. Let  $N$ be a $\Pi_1^{\Ss}(\{p_1^{\Ss},x_{\cone}\})$-iterable $X$-premouse
 which is sound and projects to $X$. Let $M=M^{\Pg}_{\ladder}(X)$. Then either
 $N\pins M|\om_1^M$ or $M|\om_1^M\pins N$.
 \end{lem}
 \begin{proof}
 Suppose not. 
  We attempt to compare $M|\om_1^M$ with $N$,
  with trees $\Tt,\Uu$ respectively (padded in the usual way for comparison).
  Note that by the contradictory hypothesis,
 the comparison is non-trivial,
 but cannot terminate successfully.
 Since both trees produce $\Pgap$-correct models as far as the tree is $\Gammagap$-short
 and $\Pgap$-transcendent,
 both trees are in fact $\Pgap$-transcendent as far as they are $\Gammagap$-short. So the comparison must reach a limit stage $\lambda$ (and let us say that $(\Tt,\Uu)$ has length $\lambda$
 such that if $\Tt$ is cofinally (in $\lambda$) non-padding, then $\Tt$ is $\Gammagap$-maximal,
 and likewise for $\Uu$. Let $b=\Sigma_M(\Tt)$.

 \begin{clm}$E^\Uu_\alpha\neq\emptyset$
  for cofinally many $\alpha<\lambda$.
 \end{clm}
\begin{proof}
 Suppose not.
 So we can fix $\alpha<\lambda$ such that $M^\Uu_\alpha=M^\Uu_{\alpha'}$ for all $\alpha'\in[\alpha,\lambda)$.
 Then note that $\Pgap(M(\Tt))\ins M^\Tt_b$ and $\Pgap(M(\Tt))\ins M^\Uu_\alpha$. We can't have $\Pgap(M(\Tt))\pins M^\Tt_b$,
 by smallness (since $\delta(\Tt)$ is a cardinal in $M^\Tt_b$, by taking appropriate hulls and using condensation, we get a $\Pg$-ladder $N\pins M^\Tt_b$).
 So $M^\Tt_b=\Pgap(M(\Tt))$.
 Similarly, $M^\Uu_\alpha=\Pgap(M(\Tt))$, but then the comparison succeeded, contradiction.
 \end{proof}
 
 So we can apply clause \ref{item:Pi_1-it_maximal_clause} of $\Pi_1$-iterability to $\Uu$.
 Let $b_n$ be a branch and $Q_n\ins M^{\Uu}_{b_n}$ witness the requirement for $n<\om$.
 Note that by smallness, $Q_n=M^{\Uu}_{b_n}$
 and $\rho_1(M^\Uu_{b_n})<\delta(\Uu)$. Let $\xi_n$
 be least such that $\xi_n+1\in b_n$ and $\rho_1(M^\Uu_{b_n})\leq\crit(E^\Uu_{\xi_n})$.
 Let $\kappa_n=\crit(E^\Uu_{\xi_n})$.
 
 \begin{clm} We have:
 \begin{enumerate}\item \label{item:rho<delta(Uu)}
  $\rho=\sup_{n<\om}\rho_1(M^\Uu_{b_n})<\delta(\Uu)$.
  \item\label{item:no_drop_after_xi_n+1} For each $n<\om$, we have $(\xi_n+1,b_n)\cap\dropset^\Uu=\emptyset$ and $\deg^\Uu_{\xi_n+1}=\deg^\Uu_{b_n}=0$, so $i^{*\Uu}_{\xi_n+1,b_n}:M^{*\Uu}_{\xi_n+1}\to M^\Uu_{b_n}$ exists and is a $0$-embedding.
  \item\label{item:sup_kappa_n<delta(Uu)}  $\sup_{n<\om}\kappa_n<\delta(\Uu)$.
 \end{enumerate}
 \end{clm}
\begin{proof}
Part \ref{item:rho<delta(Uu)}: Otherwise, condensation easily gives that
 $\Pgap(M(\Uu))$ is a $\Pgap$-ladder, but since $\Pgap(M(\Uu))\ins M^\Tt_b$,  this contradicts smallness.
 
Part \ref{item:no_drop_after_xi_n+1}:
This is a routine consequence of the fact that $\rho_1(M^\Uu_{b_n})\leq\kappa_n$.

 Part \ref{item:sup_kappa_n<delta(Uu)}: Suppose not. With $\rho$ as in part \ref{item:rho<delta(Uu)}, let $\beta$
 be least such that $\rho<\nu(E^\Uu_\beta)$
 and let $n$ be such that $\nu(E^\Uu_\beta)\leq\kappa_n$ ($n$ exists as $\sup_{n<\om}\kappa_n=\delta(\Uu)$).
 So $\beta<\pred^\Uu(\xi_n+1)$.
 Let $\eta$ be least such that $\eta\geq\beta$ and $\eta+1\in b_n$. So $\pred^\Uu(\eta+1)\leq\beta$,
 so $\eta<\xi_n$, so \[\crit(E^\Uu_\eta)<\rho_1(M^\Uu_{b_n})\leq\rho<\nu(E^\Uu_\beta)\leq\nu(E^\Uu_\eta).\]
 This contradicts the preservation of fine structure along $b_n$ (irrespective of drops along $b_n$). 
\end{proof}

 \begin{clm} $c=\liminf_{n<\om}b_n$ is a  $\Uu$-cofinal branch, $c\cap\dropset^\Uu$ is finite, and $M^\Uu_c=\Pgap(M(\Uu))$.\end{clm}
 \begin{proof}
Let $m<\om$ with $\gamma_m^{\Pgap(M(\Uu))}>\sup_{n<\om}\kappa_n$,
 and let $ m\leq m_0<m_1<\om$.
 Since
 $T_{m_0}^{M^\Uu_{b_{m_0}}}=T_{m_0}^{\Pg(M(\Uu))}=T_{m_0}^{M^\Uu_{b_{m_1}}}$, we have that
 $\rg(i^{*\Uu}_{\xi_{m_i+1},b_{m_i}})\cap\gamma_{m_0}^{\Pg(M(\Uu))}$ is cofinal in $\gamma_{m_0}^{\Pg(M(\Uu))}$,
 for $i\in\{0,1\}$.  So letting $\eta_i+1\in b_{m_i}$ be least with $\gamma_{m_0}^{\Pg(M(\Uu))}\leq\crit(E^\Uu_{\eta_i})$,
 for $i\in\{0,1\}$, 
 then $\pred^\Uu(\eta_0+1)=\pred^\Uu(\eta_1+1)\in b_{m_0}\cap b_{m_1}$, by the Zipper Lemma. This shows that  $c=\liminf_{n<\om}b_n$ is a $\Uu$-cofinal branch. The agreement between the branches $b_n$ (for sufficiently large $n$) also gives that $c\cap\dropset^\Uu$ is bounded in $\lambda$ and hence finite, since $\rho_1(M^\Uu_{b_{m_0}})\leq\rho$ for each $m_0<\om$.
 
 Now note that for sufficiently large $\xi\in c$,
 letting $\delta_\xi$
 be the least Woodin of $M^\Uu_\xi$ with $\delta_\xi>\rank(X)$, we have:
 \begin{enumerate}[label=--]\item $(\xi,c)^\Uu\cap\dropset^\Uu=\emptyset$ and $\deg^\Uu_\xi=0$,
 \item $M^\Uu_\xi$ is $\delta_\xi$-sound with $\rho_1(M^\Uu_\xi)<\delta_\xi$,
 \item $\sup_{n<\om}\alpha_n^{M^\Uu_\xi}=\OR(M^\Uu_\xi)$ and
  $\sup_{n<\om}\gamma_n^{M^\Uu_\xi}=\delta_\xi$,
  \item $i^\Uu_{\xi c}``\OR(M^\Uu_\xi)$ is cofinal in $\OR(M^\Uu_c)$ and $i^\Uu_{\xi c}(\alpha_n^{M^\Uu_\xi})=\alpha_n^{M^\Uu_c}$ for all $n<\om$,
  \item $i^\Uu_{\xi c}$ is continuous at $\delta_\xi$ and $i^\Uu_{\xi c}(\gamma_n^{M^\Uu_\xi})=\gamma_n^{\Pg(M(\Uu))}$ for all $n<\om$,
  and so $i^\Uu_{\xi c}(\delta_\xi)=\delta(\Uu)\in\wfp(M^\Uu_c)$.
 \end{enumerate}
 Note that for each $n<\om$, 
 for sufficiently large $\xi\in c$, we  have
 $\gamma_n^{M^\Uu_\xi}<\crit(i^\Uu_{\xi c})$
 and
  $T_n^{M^\Uu_c}=T_n^{M^\Uu_\xi}=T_n^{\Pg(M(\Uu))}$.
 Also $M^\Uu_c$ is $\delta(\Uu)$-sound,
 and it follows that $M^\Uu_c=\Pgap(M(\Uu))$,
 completing the proof of the claim.
 \end{proof}
 
 Since $\Pgap(M(\Uu))\ins M^\Tt_{\Sigma_M(\Tt)}$
 (or $M^\Tt_\alpha$ for sufficiently large $\alpha<\lambda$, if $\Tt$ is eventually only padding), we have found a successful comparison,
 a contradiction, completing the proof of the lemma.
 \end{proof}

Now let $X\in\mathscr{C}$ and $M=M^{\Pg}_{\ladder}(X)$. Let $\bar{\beta}$ and $\pi:\Ss_{\bar{\beta}}(\RR^M)\to\Ss$ be the $\Sigma_1$-elementary embedding described in \cite{gaps_as_derived_models}.
So $\RR^M=\Ss_{\bar{\beta}}(\RR^M)\cap\RR$.
We now want to show that
 $\Sigma_2^{\Ss}(\{x_{\cone},X\})\rest\RR^M$ is $\Pi_2^{\Ss_{\bar{\beta}}(\RR^M)}(\{x_{\cone},X\})$; this is analogous to what we showed for projective-like gaps of countable cofinality.

 Let $\varphi$ be $\exists^\RR\Pi_1$,
 of form $\varphi(u)\iff \exists^\RR w\ \psi(w,u)$, where $\psi$
 is $\Pi_1$.
 Let $x\in\RR^M$.
 For $n<\om$,
 let $b^M_n\in\BB_{\delta^M}^M$
 be the Boolean value of the statement\footnote{***This needs explaining what the generic $L(\RR)$ means. It's not just the naive derived model, but the model encoded as in \cite{gaps_as_derived_models}.} ``$M|\alpha^M_n$
 verifies that the generic $L(\RR)$ models $\neg\psi(\dot{w},x)$,
 where $\dot{w}$ is the generic real''. Let $S$ be the natural tree
 of attempts to build $(H,g)$
 such that $H\preccurlyeq_1 M$
 and $\delta^M,p_1^M,x_{\cone}\in H$
 and $g\sub H$,
 and letting $\bar{M}$ be the transitive collapse of $H$ and $\tau:\bar{M}\to H$ the isomorphism
 and $\bar{g}=\tau^{-1}``g$, then $\bar{g}$ is a $(\bar{M},\BB_{\delta^{\bar{M}}}^{\bar{M}})$-generic
 filter such that for no $n<\om$ is $b_n^{\bar{M}}=\tau^{-1}(b_n^M)\in \bar{g}$.
 More precisely, $S$ is the set of finite sequences $((\vec{x}_0,q_0),(\vec{x}_1,q_1),\ldots,(\vec{x}_{n-1},q_{n-1}))$ such that:
 \begin{enumerate}
 \item\label{item:x-vec_i_in_relevant_Hull} $\vec{x}_i\in \Hull_1^{M|\alpha_i^M}(\gamma_i^M\cup X^M\cup\{p_1^M\})^{<\om}$ for $i<n$,
 \item  if $n>0$ then
 $\vec{x}_0$ includes $p_1^M$, $x_0$\footnote{Of course, we had assumed that $x_0=\emptyset$, but in general this should be included.} and $\delta^M$, 
 \item 
  $q_i\in \vec{x}_i$
  for all $i<n$,
  \item\label{item:alpha_i_gets_in} $\alpha_i^M\in \vec{x}_{i+1}$ for all $i+1<n$,
  \item \label{item:Sigma_1-elem_closure}
 ($\Sigma_1$-elementarity closure) for all $i<n$, for the first $i$ 3-tuples
 (in some reasonable recursive ordering) of form
 $(\varrho,\vec{b},k)$
such that:
\begin{enumerate}[label=--]\item $\varrho=\varrho(\vec{v},u)$ is a
$\Sigma_1$ formula in the language of passive premice with free variables exactly $\vec{v},u$,
\item 
 $\vec{b}$ is a finite tuple with $\vec{b}\sub\vec{x}_0\conc\ldots\conc\vec{x}_{i-1}$,
 \item $\lh(\vec{v})=\lh(\vec{b})$,
 \item 
 $k<i$,
 \item $\vec{b}\sub M|\alpha_k^M$,
 and
 \item 
  $M|\alpha_k^M\sats\exists u\ \varrho(\vec{b},u)$,
  \end{enumerate}
there is some  $y\in( M|\alpha_k^M)\cap \vec{x}_{i}$
 such that $M|\alpha_k^M\sats\varrho(\vec{b},y)$,
  \item $q_i\in\BB_\delta^M$ for $i<n$,
 \item\label{item:meet_next_antichain}
 if $0<i<n$
 and there is a maximal antichain  
 $A$ of
  $\BB^M_{\delta^M}$ with $A\in\vec{x}_0\conc\ldots\conc\vec{x}_{i-1}$
  and $A\in M|\gamma_{i-1}^M$ such that for no $j<i$ is $q_j\in A$, then  letting $A_i$ be the first such one
  appearing in the list
  $\vec{x}_0\conc\ldots\conc\vec{x}_{i-1}$, we have $q_i\leq q_{i-1}$ and $q_i\in A$;
  otherwise $q_i=q_{i-1}$,
  \item\label{item:avoid_b_j} for no $j+1<n$
 is $q_j\leq b_j^M$.\footnote{Note that $b_j^M\leq b_{j+1}^M$ for all $j$, and on the other hand, we have $q_{j+1}\leq q_j$.}
 \end{enumerate}
 
 Note that in condition \ref{item:meet_next_antichain},
 since $A\in M|\gamma_{i-1}^M$,
 and by \S\ref{sec:antichains},
 the the fact that $A$ is a maximal antichain is a
 simple condition
 on the parameters $A,M|\delta^M$, and moreover, 
 there are cofinally many $\eta<\gamma_{i-1}^M$
 such that $M|\eta\preccurlyeq M|\delta^M$,
 so 
  it is just a simple fact about
 $\Th_{\rSigma_1}^{M|\alpha_{i-1}^M}(\vec{x}_0\conc\ldots\conc\vec{x}_{i-1})$.

 It is straightforward to see that if $\left<(\vec{x}_i,q_i)\right>_{i<\om}$ is an infinite branch through $S$ and $X=\bigcup_{i<\om}\vec{x}_i$
 and $g$ is the filter generated by $\{q_i\bigm|i<\om\}$,
 then letting $\bar{M}$ be the transitive collapse of $X$ and
 $\pi:\bar{M}\to M$ the uncollapse and $\bar{g}=\pi^{-1}``g$,
 then $\bar{M}=\Pgap(\bar{M}|\bar{\delta})$
 and $\bar{\delta}$ is Woodin in $\bar{M}$  and $\bar{g}$ is $(\bar{M},\BB^{\bar{M}}_{\delta^{\bar{M}}})$-generic
 and letting $w$ be the generic real, $\bar{M}[\bar{g}]\sats$``$\psi(w,x)$ holds in the generic $L(\RR)$'', which proves that $\J_{\beta}(\RR)\sats\psi(x,w)$
 (since we can iterate just at the Woodins strictly above $\delta^{\bar{M}}$ to produce $\J_\beta(\RR)$ as a derived model). 
 
 Now suppose that $S$ is instead wellfounded. We claim that $\J_\beta(\RR)\sats\all^\RR w\ \neg\psi(w,x)$''.
 The proof will also show that
 there is a ``$\all^\RR w\ \neg\psi(w,x)$-witness''
 which is a proper segment of $M|\om_1^M$ which projects to $\om$.

 For $n<\om$, an \emph{$n$-partial-$\Pgap$-ladder} is a pair $(N,P)$ such that $N$ is a premouse and $P\pins N$ and
  there are $\eta_0<\ldots<\eta_n<\OR^N$
 which are limit cardinals of $N$,
 $\eta_n^{++N}<\OR^N$
 and $\eta_n^{++N}$ is the largest cardinal of $N$,
 $N$ is $\Pgap$-closed,
 and there are $P_0\pins P_1\pins\ldots\pins P_n=P$
 such that $\eta_i<\OR^{P_i}$
 and $\rho_1^{P_i}=\rho_\om^{P_i}=\eta_i$
 and $P_i$ is a relevant premouse
 with $\eta_j=\gamma_j^{P_i}$ for $j\leq i$,
 and $(\eta_j,P_j)$
 is the $j$th $\Pgap$-rung of $P_i$ for $j\leq i$.
 (Let's change the definition of a rung to require that the $\Sigma_1$-theories of $P|\alpha_n^P$ in parameters $<\gamma_n^P$ and in $X^P$ 
 and $p_1^P$ are identical.)
 Note that $\eta_0,\ldots,\eta_n$
 and $P_0,\ldots,P_n$ are all determined by $P$.
  
  Note that given an $n$-partial-$\Gammagap$-ladder $(N,P)$,
  and $\eta_0,\ldots,\eta_n$, $P_0,\ldots,P_n$ the corresponding objects,
  for each $i\leq n+1$
  and $j\leq i+1$, we can define $S_j^{P_i}$ just as above, except that this includes only sequences $\sigma=((\vec{x}_0,q_0),\ldots,(\vec{x}_{k-1},q_{k-1}))$ of length $k\leq j$.
  For $k\leq j$,
  let $S_j^{P_i}\rest k=S_k^{P_i}$ (its restriction to sequences $\sigma$ of length $\leq k$).

  For $i<n$,
  note that $S_{i+1}^{P_i}$ is isomorphic to $S_{i+1}^{P_n}$, and letting $\pi_{in}:S_{i+1}^{P_i}\to S_{i+1}^{P_n}$ be the natural isomorphism,  $\pi_{in}$ does not move any ordinals $<\gamma_i^{P_i}=\gamma_i^{P_n}$. (For $j+1<i+1$,
  so $j<i$, we have $b_j^{P_i}=b_j^{P_n}$, so $P_i$ will agree with $P_n$ about whether $q_j\leq b_j$,
  ensuring agreement regarding condition \ref{item:avoid_b_j}).

  Let $N$ be a premouse
  with $x\in\RR^N$ 
  and let $S,\sigma\in N$
  and $n<\om$.
  We say that $(N,P,\Delta)$ is a \emph{$(\all^\RR\neg\psi(x),\sigma)$-prewitness} iff 
  $(N,P)$ is an $n$-partial-$\Gammagap$-ladder,
  as witnessed by $\eta_0,\ldots,\eta_n,P_0,\ldots,P_n=P$,
  $\sigma=\left<(\vec{x}_i,q_i)\right>_{i<n}\in S_n^{P_n}$
  and  $\Delta\in N|\eta_n^{++N}$ is a non-empty tree (set of finite sequences closed under initial segment, so $\emptyset\in\Delta$), whose elements take the form
  \[ \tau=((N_{n+1},P_{n+1},\sigma_{n+1},\Delta_{n+1}),\ldots,(N_{n+k},P_{n+k},\sigma_{n+k},\Delta_{n+k})), \]
  and moreover (writing $N_n=N$, $P_n=P$, $\sigma_n=\sigma$, $\Delta_n=\Delta$), for all $\tau\in\Delta$ as above,
  \begin{enumerate}
   \item $(N_{n+i+1},P_{n+i+1})$
   is an $(n+i+1)$-partial-$\Pgap$-ladder with $P_{n+i}=P_{n+i}^{P_{n+i+1}}$,
   \item $P_{n+i}\pins P_{n+i+1}\pins N_{n+i+1}\pins N_{n+i}$,
   \item $\rho_1^{N_{n+i+1}}=\rho_\om^{N_{n+i+1}}=(\eta_{n+i}^{(N_{n+i},P_{n+i})})^{+N_{n+i}}$,
   \item $\sigma_{n+i+1}\in S_{n+i+1}^{P_{n+i+1}}$
   and $\pi_{n+i}^{P_{n+i},P_{n+i+1}}(\sigma_{n+i})=\sigma_{n+i+1}\rest (n+i)$,
   \item $\Delta_{n+i+1}\in N_{n+i+1}|(\eta_{n+i+1}^{(N_{n+i+1},P_{n+i+1})})^{++N_{n+i+1}}$ (recall this is the largest cardinal proper segment of $N_{n+i+1}$),
   \item $\Delta_{n+i+1}=(\Delta_{n+i})_{\tau\rest (i+1)}$,
   \item\label{item:every_sigma'_handled} for every $\sigma'\in S_{n+i+1}^{P_{n+i}}$
   such that $\sigma_{n+i}\pins\sigma'$ (so $\lh(\sigma')=n+i+1$)
   there is $\tau'\in\Delta$
   such that $\tau'\rest i=\tau\rest i$
   and $\lh(\tau')=i+1$
   and letting $\tau'(i)=(N'',P'',\sigma'',\Delta'')$, then $\pi^{P_{n+i},P''}_{n+i+1}(\sigma')=\sigma''$.
  \end{enumerate}

  Suppose that $(N,P,\Delta)$ is a $(\all^\RR\neg\psi(x),\emptyset)$-prewitness and $N$ is iterable. 
  Then we claim that $\Jj_\beta(\RR)\sats\all^\RR u\ \neg\psi(x,u)$.
  Suppose not and let $w\in\RR$
  be such that $\Jj_\beta(\RR)\sats\psi(x,w)$. Fix a surjection $f:\om\to X^M$. We begin to iterate $N$ to make $w$ generic for the image of the extender algebra of $P$ at $\delta^P$, producing iteration tree $\Tt$,
  except that at the first stage
  $\beta<\lh(\Tt)$
  such that there are no $w$-bad extender algebra axioms indexed below $i^\Tt_{0\beta}(\gamma_0^P)$, we proceed as follows. We will specify some $\sigma\in S_1^{i^{\Tt}_{0\beta}}(P)$. Recall that elements of $S_1$ have length $1$,
  taking the form $\sigma=((\vec{x}_0,q_0))$.
  Note that conditions
  \ref{item:alpha_i_gets_in}, \ref{item:Sigma_1-elem_closure},
  \ref{item:meet_next_antichain}
  and \ref{item:avoid_b_j}
  are trivial for elements of length $1$.
  So  just set $
  q_0$ to be the ``$1$''
  of the extender algebra of $i^\Tt_{0\beta}(P)$ at $\delta^{i^\Tt_{0\beta}(P)}$)
  and $\vec{x}_0=(p_1^M,x_0,\delta^M,q_0,f(0),A)$
  where $f:\om\to X$ is the surjection fixed earlier, 
  and $A$ is some maximal antichain of $\BB^{P_\beta}_{\delta^{P_\beta}}$ with $A\in P_\beta||\gamma_0^{P_\beta}$;\footnote{There is actually no need to put $A$ in $\vec{x}_0$,
  but it does no harm. We only put it in for expository reasons;
  it forces us to deal with some maximal antichains earlier in the overall process than we would otherwise have to. There is also no need to put $f(0)$ in, but we will do something similar in general to arrange that we include all elements of $X$ along any infinite branch, mostly for convenience.}
  it is easy to see that $(\vec{x}_0,q_0)\in S_1^{i^{\Tt}_{0\beta}(P)}$. 
  
  Now since $(N,P,\Delta)$ is a $(\all^\RR\neg\psi(x),\emptyset)$-prewitness, so is $(M_\beta,P_\beta,\Delta_\beta)=(M^\Tt_\beta,i^\Tt_{0\beta}(P),i^\Tt_{0\beta}(\Delta))$. So we can apply property \ref{item:every_sigma'_handled}
  to $\sigma'=((\vec{x}_0,q_0))\in S_1^{P_\beta}=S_{n+i+1}^{P_{\beta,n+i}}$
  where $n=i=0$ and $P_{\beta,0}=P_\beta$,
  and $\sigma_{n+i}=\sigma_0=\emptyset$.
  So fix some $\tau'\in \Delta_\beta$ with $\lh(\tau')=1$ and letting $\tau'(0)=(N'',P'',\sigma'',\Delta'')$, we have $\pi^{P_\beta,P''}_1((\vec{x}_0,q_0))=\sigma''$.

  We choose $E^\Tt_\beta$
  as for the $w$-genericity iteration of $P''$ at $\delta^{P''}$,
  as long as this yields $\lh(E^\Tt_\beta)<\gamma_1^{P''}$. Suppose that this does indeed occur.
  Note that $P_\beta|\gamma_0^{P_\beta}=P''|\gamma_0^{P''}$ and $\gamma_0^{P_\beta}$ is a limit cardinal of $N_\beta=M^\Tt_\beta$, $P_\beta$, $N''$ and $P''$.
  And recall that there is no $w$-bad extender algebra axiom induced by an extender indexed in $\es^{P_\beta}|\gamma_0^{P_\beta}$,
  so if $E^\Tt_\beta$
  exists then $\gamma_0^{P_\beta}=\gamma_0^{P''}<\lh(E^\Tt_\beta)<\delta^{P''}$.
  Recall that also $\rho_\om^{P_\beta}=\gamma_0^{P_\beta}=\eta_0^{(N_\beta,P_\beta)}$,
  and $\rho_\om^{N''}=(\eta_0^{(N_\beta,P_\beta)})^{+N_\beta}=(\gamma_0^{P_\beta})^{+N_\beta}$ and $\eta_1^{(N'',P'')}=\gamma_1^{P''}$ is a limit cardinal of both $N''$ and $P''$ and $N''|\eta_1^{(N'',P'')}=P''|\gamma_1^{P''}$, and $\rho_\om^{P''}=\gamma_1^{P''}$, and since $\nu(E^\Tt_\beta)$
  is a $P''$-cardinal and $\gamma_0^{P''}<\lh(E^\Tt_\beta)<
  \gamma_1^{P''}$, if $\gamma\geq\beta$
  and $\pred^\Tt(\gamma+1)=\beta$
  then we will have:
  \begin{enumerate}[label=--]\item if $\crit(E^\Tt_\gamma)<\gamma_0^{P_\beta}$ then $M^{*\Tt}_{\gamma+1}=N_\beta$ (so $\gamma+1\in\dropset^\Tt$), and
   \item if $\crit(E^\Tt_\gamma)\geq\gamma_0^{P_\beta}$
   then $M^{*\Tt}_{\gamma+1}=N''$
   (so $\gamma+1\notin\dropset^\Tt$).
  \end{enumerate}
 Let $\gamma+1$ be such,
  and let $\gamma+1\leq^\Tt\vareps$ 
  with $(\gamma+1,\vareps]^\Tt\cap\dropset^\Tt=\emptyset$. If $\crit(E^\Tt_\gamma)<\gamma_0^{P_\beta}$, then we proceed making $w$ generic for the extender algebra of the image $P_\vareps$ of $P$ at $\delta^{P_\vareps}$, unless there are no $w$-bad axioms induced by extenders in $\es(P_\vareps||\gamma_0^{P_\vareps})$,
  in which case we will proceed like at stage $\beta$.
  If instead $\gamma_0^{P_\beta}\leq\crit(E^\Tt_\gamma)$,
  then we proceed by making $w$ generic for the extender algebra of $P_{\vareps}=i^{*\Tt}_{\gamma+1,\vareps}(P'')$ at $\delta^{P_\vareps}$,
  unless there are no $w$-bad extenders in $\es(P_\vareps||\gamma_1^{P_\vareps})$.
  
  Now suppose that $\gamma+1,\vareps$ are as above,
  with $\gamma_0^{P_\beta}\leq\crit(E^\Tt_\gamma)$,
  and there are indeed no $w$-bad axioms induced by extenders
  in $\es^{P_\vareps}||\gamma_1^{P_\vareps}$.
  We proceed at stage $\vareps$ in a fashion much like at stage $\beta$ earlier, but this time there is more to consider,
  most importantly regarding condition \ref{item:meet_next_antichain}
  and (sort of) condition \ref{item:avoid_b_j}.
  Consider condition \ref{item:meet_next_antichain}.
 If possible, we want to find an element $\sigma'_\vareps\in S_2^{P_\vareps}$
  which is consistent with making $w$ generic and such that $\sigma_\vareps\pins\sigma'_\vareps$, where $\sigma_\vareps=i^{*\Tt}_{\gamma+1,\vareps}(\sigma'')$.
  For this, we must consider conditions \ref{item:meet_next_antichain}
  and  \ref{item:avoid_b_j}.
  Since we are interested in $\lh(\sigma_\vareps)=2$,
  condition  \ref{item:meet_next_antichain}
  is non-trivial. Let $\sigma_\vareps=((\vec{x}_0,q_0))$. Suppose there is a maximal antichain $A$ of $\BB^{P_\vareps}_{\delta^{P_\vareps}}$
  with $A\in\vec{x}_0$ and $A\in P_\vareps|\gamma_0^{P_\vareps}$
   such that $q_0\notin A$,
   and let $A_0$ be the first one appearing in $\vec{x}_0$.\footnote{There could be some such antichain, since we put $A$ explicitly in $\vec{x}_0$ earlier.}
  Because $w$ satisfies all extender algebra axioms induced by extenders indexed in $P_\vareps||\gamma_0^{P_\vareps}=P_\beta||\gamma_0^{P_\beta}$
  (in fact satisfies all those induced by extenders indexed in $P_\vareps||\gamma_1^{P_\vareps}$), there is exactly one $q\in A$ such that, if $\varphi\in P_\vareps||\gamma_0^{P_\vareps}$
  is any formula such that $q=[\varphi]$, then $w\sats\varphi$.
  Set $q_1=q$. We have $q_1\leq q_0$ since $q_0=1$, so we have satisfied
  the requirements for condition \ref{item:meet_next_antichain}.
  (In more generality, when selecting $q_i$, we will have ``$w\sats q_{i-1}$'' in the appropriate sense, and we will set $q_i=q\wedge q_{i-1}$,
  which will be a non-zero condition since ``$w\sats q$''
  and ``$w\sats q_{i-1}$''.)

  Now consider condition \ref{item:avoid_b_j}.\footnote{Since $q_0=1$,
  this condition is actually rather trivial in the case of interest. But let's ignore this, for expository purposes.}
  Suppose first that $q_0\leq b_0^{P_\vareps}$. (Then clearly there is no way to find an appropriate $\sigma'_\vareps$.)
  In this case we will continue with $w$-genericity iteration over the image of $P_\vareps$;
  note that then if $E^\Tt_\vareps$ exists,
  we will have $\gamma_1^{P_\vareps}\leq\nu(E^\Tt_\vareps)$.
  Since $P_\vareps\pins N_\vareps$
  and $\rho_\om^{P_\vareps}=\gamma_1^{P_\vareps}=\eta_1^{(N_\vareps,P_\vareps)}$, if $\xi\geq\vareps$
  and $\pred^\Tt(\xi+1)=\vareps$
  then:
  \begin{enumerate}
   \item if $\crit(E^\Tt_\xi)<\gamma_1^{P_\vareps}$ then $M^{*\Tt}_{\xi+1}=N_\vareps$ (so $\xi+1\notin\dropset^\Tt$), and
   \item if $\crit(E^\Tt_\xi)\geq\gamma_1^{P_\vareps}$ then $M^{*\Tt}_{\xi+1}=P_\vareps$ (so $\xi+1\in\dropset^\Tt$).
  \end{enumerate}
  For $\xi+1$ as above and $\theta$ such that $\xi+1\leq^\Tt\theta$
  and $(\xi+1,\theta]^\Tt\cap\dropset^\Tt=\emptyset$,
  if $\crit(E^\Tt_\xi)<\gamma_1^{P_\vareps}$ then we proceed at stage $\theta$ like at stages $\vareps$ as before;
  if instead $\gamma_1^{P_\vareps}\leq\crit(E^\Tt_\xi)$
  then we simply continue with $w$-genericity iteration, without further restrictions.\footnote{Note here that we must continue to consider extender algebra axioms induced by extenders overlapping $\gamma_1^{P_\vareps}$.
  This is because the assumption
  is that $q_0\leq b_0^{P_\vareps}$, but this refers to the standardly defined extender algebra, without restrictions on critical points of extenders used to induce axioms.
  If we switch at this point to only inducing axioms with extenders $E$ with $\gamma_1^{P_\vareps}\leq\crit(E)$, then we will be changing the forcing under consideration,
  so that the fact that $q_0\leq b_0^{P_\vareps}$ will become irrelevant.}
  
Now suppose instead that $q_0\not\leq b_0^{P_\vareps}$.
Then we have satisfied condition \ref{item:avoid_b_j}, 
and it is now easy to find an appropriate $\sigma'_\vareps\in S_2^{P_\vareps}$; it just remains to specify $\vec{x}_1$ appropriately. Put $\vec{x}_1=(q_1,\alpha_0^{P|\vareps},f(1),y)$
where $y$ is the least object satisfying the requirements of condition \ref{item:Sigma_1-elem_closure},
with respect to the natural wellorder of $P_\vareps$ given by
combining the enumeration $f$ of $X$ with $P_\vareps$-(prewell)order of constructibility. It is easy to see that $((\vec{x}_0,q_0),(\vec{x}_1,q_1))\in S_2^{P_\vareps}$,
and note that also $w\sats\varphi$, whenever $\varphi\in P_\vareps|\gamma_1^{P_\vareps}$
and $P_\vareps\sats$``$q_1=[\varphi]$''.

Since $(N_\vareps,P_\vareps,\Delta_\vareps)$
is a $(\all^\RR\neg\psi(x),\sigma_\vareps)$-prewitness,
applying property \ref{item:every_sigma'_handled}
to $\sigma'_\vareps\in S_2^{P_\vareps}$,
we can fix $\tau'\in\Delta_\vareps$
such that $\lh(\tau')=1$
and letting $\tau(0)=(N'',P'',\sigma'',\Delta'')$,
we have $\pi^{P_\vareps,P''}_{2}(\sigma'_\vareps)=\sigma''$.

The process described so far generalizes in the obvious manner to all stages $\theta$ of $\Tt$.
We have two kinds of stages $\theta<\lh(\Tt)$: $\theta$ can be \emph{$w$-consistent},
or \emph{$w$-inconsistent}.
Given $\theta<\lh(\Tt)$,
the set $\{\vareps\leq^\Tt\theta\bigm|\vareps\text{ is }w\text{-consistent}\}$ will be a closed initial segment of $[0,\theta]^\Tt$, including $\vareps=0$.
At each stage $w$-consistent stage $\theta<\lh(\Tt)$,
 we will  have $N_\theta=M^\Tt_\theta$,  $P_\theta\pins N_\theta$,
$n_\theta<\om$, such that
$(N_\theta,P_\theta)$
is an $n_\theta$-partial-$\Pgap$-ladder, and have
$\sigma_\theta\in S_{n_\theta}^{P_\theta}$
and $\Delta_\theta\in N_\theta$
such that $(N_\theta,P_\theta,\Delta_\theta)$ is a $(\all^\RR\neg\psi(x),\sigma_\theta)$-prewitness.
We will have relationships
and drops along $[0,\theta]^\Tt$ essentially like in the cases described above.
This is maintained by induction
using crucially the fact that for  any $\theta$, $(0,\theta]^\Tt$
drops only finitely often,
and if $\theta_0<^\Tt\theta_1$ 
then $n_{\theta_0}\leq n_{\theta_1}$, with $n_{\theta_0}<n_{\theta_1}$ iff $(\theta_0,\theta_1]^\Tt\cap\dropset^\Tt\neq\emptyset$. 
The $w$-inconsistent stages
$\vareps$ are just those such that for some $\theta<^\Tt\vareps$,
$\theta$ is $w$-consistent,
and at stage $\theta$,
we want to extend $\sigma_\theta$
to some $\sigma'_\theta$
with $\lh(\sigma'_\theta)=n_\theta+1$,
but cannot, because we violate condition \ref{item:avoid_b_j}
and $\crit(E^\Tt_\xi)\geq\gamma_{n_\theta}^{P_\theta}$,
where $\xi+1=\min((\theta,\vareps]^\Tt)$.
That is, regarding the violation of condition  \ref{item:avoid_b_j},  there is some $j+1<n_\theta+1$ (so $j< n_\theta$, so $0<n_\theta$) such that $q_{\theta j}\leq b^{P_\theta}_j$, where $\sigma_\theta=((\vec{x}_{\theta0},q_{\theta0}),\ldots,(\vec{x}_{\theta,(n_\theta-1)},q_{\theta,(n_\theta-1)}))$. We will have here that $w\sats\varphi$,
whenever $\varphi\in P_\theta|\gamma_{n_\theta}^{P_\theta}$
and $P_\theta\sats$``$q_{\theta j}=[\varphi]$''. In this case we will also have that $(\xi+1,\vareps]^\Tt\cap\dropset^{\Tt}=\emptyset$,
and we will be iterating to make $w$ generic for the extender algebra of $M^{\Tt}_\vareps$
at $\delta^{M^{\Tt}_\vareps}$
(and likewise at stage $\theta$,
for the extender algebra of $P_\vareps$ at $\delta^{P_\vareps}$). 

Now the construction of $\Tt$ must terminate in countably many steps.
For suppose we reach a construction of length $\om_1+1$.
Then there are only finitely many drops along $[0,\om_1)^\Tt$.
But then the usual argument for termination of genericity iteration yields a contradiction.

Let $\vareps+1=\lh(\Tt)$.
Suppose  $\vareps$ is $w$-inconsistent, and adopt the notation with $\theta,\xi$, etc, as above.
Since the process terminates,
$w$ is generic over $M^\Tt_\vareps$ for $\BB^{M^\Tt_\vareps}_{\delta^{M^\Tt_\vareps}}$.
But $w\sats\varphi$
where $\varphi\in P_\theta|\gamma_{n_\theta}^{P_\theta}$
and $P_\theta\sats$``$q_{\theta_j}=[\varphi]$''.
Since $i^{*\Tt}_{\xi+1,\vareps}:P_\theta\to M^\Tt_\vareps$
is sufficiently elementary
and
$\gamma_{n_\theta}^{P_\theta}\leq\crit(i^{*\Tt}_{\xi+1,\vareps})$, we also have $M^\Tt_\theta\sats$``$q_{\theta_j}=[\varphi]$''.
So $q_{\theta_j}\in g_w$,
where $g_w$ is the $M^\Tt_\theta$-generic filter determined by $w$.
But $q_{\theta j}\leq b_j^{P_\theta}$, and so $q_{\theta j}\leq b_j^{M^\Tt_\vareps}$,
so $b_j^{M^\Tt_\vareps}\in g_w$,
which gives that $\Jj_\beta(\RR)\sats\neg\psi(x,w)$,
a contradiction.

So $\vareps$ is $w$-consistent.
So $(N_\vareps,P_\vareps,\sigma_\vareps,\Delta_\vareps)$
are defined, etc.
Because the process terminates at this stage, there are no $w$-bad extender algebra axioms induced by extenders indexed in $\es^{P_\vareps|\gamma_{n_\vareps}^{P_\vareps}}$.
So 
we try to define $\sigma_\vareps'$. If this attempt violates condition \ref{item:avoid_b_j},
then we proceed at this stage
with looking for $E^\Tt_\vareps$
within $P_\vareps$.
But since the process terminates at this stage, it is just that $w$ is already generic over $P_\vareps$. But this now yields the same contradiction as in the preceding paragraph.
So we must successfully define $\sigma'_\vareps$.
But there is no $w$-bad extender found in $P_{\vareps,1}$ indexed below $\gamma_{n_\vareps+1}^{P_{\vareps,1}}$,
so we immediately attempt to extend $\sigma'_\vareps$
one step further to $\sigma''_\vareps$. This must satisfy condition \ref{item:avoid_b_j}, since otherwise we get a contradiction like before.
So we get $\sigma''_\vareps$.
Etc. But this produces $N_{\vareps,n}$ for all $n>0$,
with $N_{\vareps,n+1}\pins N_{\vareps,n}$, a contradiction.

\bibliographystyle{plain}
\bibliography{../bibliography/bibliography}
\end{document}